\documentclass[10pt,a4paper,tbtags,leqno]{amsart} 
\usepackage[top=15mm,bottom=15mm,left=15mm,right=15mm]{geometry} 
\usepackage[dvipdfmx]{graphicx} 
\usepackage{amsmath,amssymb,amsthm} 
\usepackage{mathtools} 
\mathtoolsset{showonlyrefs=true} 
\usepackage{bm} 
\usepackage{comment} 
\usepackage{float} 
\usepackage{ascmac} 
\usepackage{color} 
\usepackage{cite} 
\usepackage{enumerate} 
\usepackage{multirow} 
\usepackage{mathrsfs} 
\usepackage{empheq} 
\usepackage{pgfplotstable} 
\usepackage{booktabs} 
\usepackage{pgfplots} 
\pgfplotsset{compat=1.18} 
\newtheorem{theorem}{Theorem}[section] 
\newtheorem{lemma}[theorem]{Lemma} 
\numberwithin{equation}{section} 
\usepackage{apptools} 
\usepackage[colorlinks=true, citecolor=blue, linkcolor=blue]{hyperref} 
\title[Semi-Lagrangian schemes for dispersive conservation laws]{Error estimates for semi-Lagrangian schemes with higher-order interpolation for conservation laws with dispersive terms} 
\author{Haruki Takemura} 
\email{takemura@ms.u-tokyo.ac.jp} 
\date{} 
\DeclareMathOperator{\cn}{cn} 
\def\norm#1#2{\left\|#2\right\|_{#1}} 
\def\snorm#1#2{\left|#2\right|_{#1}} 
\newcommand{\relmiddle}{\mathrel{}\middle|\mathrel{}} 
\newcommand{\pxa}{0.002} 
\newcommand{\pya}{0.0004} 
\newcommand{\pxb}{0.01} 
\newcommand{\pyb}{0.00005} 
\begin{document} 
\begin{abstract} 
  We establish error estimates for semi-Lagrangian schemes for the initial value problem of one-dimensional conservation laws with a dispersive term, including the Korteweg--de Vries equation. The schemes considered in this paper are based on the semi-Lagrangian technique combined with spatial discretization by higher-order interpolation operators. For the semi-Lagrangian schemes equipped with the spline or Hermite interpolation operators of order $ 2 s - 1 $, we derive an $L^2$-error estimate of $ O (\Delta t^r + h^{2s} / \Delta t) $ and an $ H^s $-error estimate of $ O (\Delta t^r + h^{s} / \sqrt{\Delta t}) $, where $ h $ and $ \Delta t $ denote the spatial mesh size and the time step size, respectively, and $ r \in \lparen 0, 1\rbrack $ is a parameter determined by the discretization of the dispersive term. A key step in the analysis is to establish the stability of the interpolation operators. Under suitable assumptions, interpolation operators of order $ 2s - 1 $ are stable with respect to the $ H^s $-norm as well as a weighted $ H^s $-norm. The weighted $H^s$-norm depends on $h$ and $\Delta t$, and it reduces to the $L^2$-norm in the limit $ h \to 0 $. 
\end{abstract} 

\maketitle 

\section{Introduction} 
 
The initial value problem of the dispersive conservation law is given by 
\begin{empheq}[left={\empheqlbrace}]{alignat=2} 
  \partial_t u + \partial_x F (u) + \nu \partial_x^3 u = 0 \quad & \text{in} \quad \mathbb{T} \times (0, T), \label{ivp01} \\ 
  u (0) = u_0 \quad & \text{in} \quad \mathbb{T}, \label{ivp02} 
\end{empheq} 
where $ \nu $ is a positive constant, $ \mathbb{T} $ denotes $ \mathbb{R} / \mathbb{Z} $, and $ F: \mathbb{R} \to \mathbb{R} $ is a sufficiently smooth function. A typical example is the Korteweg--de Vries (KdV) equation obtained by substituting $ F (u) = \frac{1}{2} u^2 $. Throughout this paper, suppose that the problem \eqref{ivp01}--\eqref{ivp02} admits a unique smooth solution. 

We consider the fully semi-Lagrangian (SL) schemes for the initial value problem \eqref{ivp01}--\eqref{ivp02} proposed in \cite{25Fe}.  Before considering schemes for \eqref{ivp01}--\eqref{ivp02}, we briefly recall the main ideas underlying fully SL approaches. The fully SL schemes have been extensively studied as numerical methods for diffusive conservation laws; see, for example, \cite{10Fe,14FaFe}. The derivation of the schemes for diffusive conservation laws is based on the probabilistic representation of solutions via the Feynman--Kac formula, as explained in \cite{02Mi,04MiTr}, where the schemes are introduced as layer methods. As emphasized in \cite{04MiTr}, although the derivation is based on stochastic discussion, the resulting schemes are entirely deterministic. Moreover, problems with the Dirichlet and Neumann boundary conditions have been treated in this framework in \cite{02MiTr,01MiTr}. 

In \cite{25Fe}, the fully SL schemes are generalized to conservation laws containing higher-order derivative terms 
\begin{empheq}[left={\empheqlbrace}]{alignat=2} 
  \partial_t u + \partial_x F (u) + \nu \partial_x^l u = 0 \quad & \text{in} \quad \mathbb{T} \times (0, T), \label{givp1} \\ 
  u (0) = u_0 \quad & \text{in} \quad \mathbb{T}, \label{givp2} 
\end{empheq} 
where $ l \geq 3 $ is an integer. For the case $ l = 3 $, two fully SL schemes were proposed as follows. Let $ N_x, N_t \in \mathbb{N} $, $ 0 = x_0 < x_1 < \cdots < x_{N_x} = 1 $, $ \Delta t = T / N $, and $ t_n = n \Delta t $. We denote the numerical approximation for $ u (x_j, t_n) $ by $ u_j^n $ and denote the vector $ (u_1^n, \ldots, u_{N_x}^n) $ by $ U^n $. Let $ h = \max_{j} |x_{j} - x_{j-1}| $, and let $ \tilde{\mathcal{I}}_h: \mathbb{R}^{N_x} \to C (\mathbb{T}) $ be an interpolation operator satisfying $ \tilde{\mathcal{I}}_h[(v_1, \ldots, v_{N_x})] (x_j) = v_j $ for all $ j \in \{1, \ldots, N_x\} $. Thus, $ \tilde{\mathcal{I}}_h[U^n]:\mathbb{T} \to \mathbb{R} $ represents the numerical solution at $ t = t_n $. In the numerical experiments of \cite{25Fe}, the linear interpolation and the cubic spline interpolation are used as $ \tilde{\mathcal{I}}_h $. In addition to these, we also consider the higher-order spline and Hermite interpolations in the present paper. 

We denote $ \partial_u F (u) $ by $ f (u) $. Then we can rewrite \eqref{ivp01}--\eqref{ivp02} as 
\begin{empheq}[left={\empheqlbrace}]{alignat=2} 
  \partial_t u + f (u) \partial_x u + \nu \partial_x^3 u = 0 \quad & \text{in} \quad \mathbb{T} \times (0, T), \label{ivp1} \\ 
  u (0) = u_0 \quad & \text{in} \quad \mathbb{T}. \label{ivp2} 
\end{empheq} 
The fully SL schemes for the problem \eqref{ivp1}--\eqref{ivp2} can be written as 
\begin{empheq}[left={\empheqlbrace}]{alignat=2} 
  u_j^0 & = u_0 (x_j), \label{slj1} \\ 
  u_j^{n} & = \sum_{(\gamma, \lambda) \in \Lambda} \gamma \tilde{\mathcal{I}}_h [U^{n-1}] (x_j - f (u_j^{n}) \Delta t + \lambda (\nu \Delta t)^{1/3}). \label{slj2} 
\end{empheq} 
In \cite{25Fe}, two choices for the parameter set $ \Lambda $ are proposed. One is given by 
\begin{align} 
  \Lambda_4 = \left\{\left(- \frac{1}{4}, \sqrt[3]{4}\right), \left(\frac{1}{4}, 0\right), \left(\frac{3}{4}, \sqrt[3]{4}\right), \left(-\frac{1}{4}, 2 \sqrt[3]{4}\right) \right\}, \label{g4} 
\end{align} 
and the other is given by 
\begin{align} 
  \Lambda_5 = \left\{ \left( \frac{3}{16}, -2 \right), \left( \frac{3}{8}, 0 \right), \left( \frac{3}{4}, 2 \right), \left( - \frac{3}{8}, 4 \right), \left( \frac{1}{16}, 6 \right) \right\}. \label{g5} 
\end{align} 
The derivation of the scheme \eqref{slj1}--\eqref{slj2} does not rely on the Feynman--Kac formula. Instead, the parameter sets $ \Lambda_4 $ and $ \Lambda_5 $ are chosen so that the resulting schemes satisfy appropriate consistency and stability conditions. See Section 3 of \cite{25Fe} for details on the construction of the parameter set $\Lambda$. As shown in \cite{25Fe}, the local truncation error of the scheme \eqref{slj1}--\eqref{slj2} with a $ q $-th order interpolation operator is $ O (\Delta t^{1/3} + h^q / \Delta t) $ for $\Lambda_4$ and $O(\Delta t^{2/3} + h^q / \Delta t)$ for $ \Lambda_5 $. 

In the literature, both explicit and implicit fully SL schemes have been studied (see, for example, \cite{00MiTr,02Mi}). However, in this paper, we restrict ourselves to the implicit one, which enables us to treat solutions with large gradients. Due to the implicit nature of the scheme, we use an iteration method to solve \eqref{slj2} for $ u_j^n $. However, this does not cause significant difficulty, since the unknown of \eqref{slj2} is a single scalar $ u_j^n $ for each $ (j,n) \in \{1,\ldots,N_x\} \times \{1,\ldots,N_t\} $. For fixed $ n $, once $ U^{n-1} $ has been obtained, the values $ u_j^n $ ($ j \in \{1,\ldots,N_x\} $) can be computed in parallel. 

SL schemes possess advantages that are well suited for nonlinear dispersive conservation laws. As a first example, SL methods can use relatively large time steps. Moreover, according to the numerical experiments reported in \cite{25Fe}, the implicit SL scheme with cubic spline interpolation is more efficient than certain finite difference schemes, that is, it achieves smaller errors for the same CPU time. 

Several types of interpolation have been employed in the construction of SL schemes. SL schemes with spline interpolation have been discussed in a number of earlier studies \cite{76Pu,90Be,99SoRoBeGh}. Whereas spline interpolation is global, one may use local interpolations such as the Hermite interpolations. The cubic interpolated pseudo-particle (CIP) scheme is a SL scheme that employs the cubic Hermite interpolation. Early works \cite{85TaNiYa,87TaYa} studied the CIP schemes for advection equations and hydrodynamic equations. A CIP scheme for the KdV equation was later proposed in \cite{91YaAo}, although it is different from \eqref{slj1}--\eqref{slj2}. Even for diffusive conservation laws, there are only a few theoretical results concerning fully SL schemes equipped with the Hermite interpolation operators. Mathematical analysis of the CIP scheme for the linear advection equation has been conducted in \cite{02NaTa,15TaFuNiIs,24KaTa}. Convergence of SL schemes with Hermite interpolations for Vlasov--Poisson equations was investigated in \cite{08Be}. 

In the present paper, we establish $ L^2 $- and $ H^s $-error estimates for the scheme \eqref{slj1}--\eqref{slj2} with higher-order interpolation operators. We impose a few conditions on the interpolation operator $ \tilde{\mathcal{I}}_h $, and these conditions are satisfied when the spline or Hermite interpolations are employed. In our previous paper \cite{25Ta}, we derived error estimates for fully SL schemes for diffusive conservation laws. For the scheme \eqref{slj1}--\eqref{slj2}, we can adopt a similar strategy. A key lemma asserts that the spline or Hermite interpolations of order $ 2 s - 1 $ are stable with respect to the $ H^s $-norm and the weighted $ H^s $-norm defined later. We note that the weighted $ H^s $-norm becomes the $ L^2 $-norm in the limit $ h \to 0 $. 

This paper is organized as follows. In Section \ref{sec:conti}, we reformulate the scheme \eqref{slj1}--\eqref{slj2} in terms of continuous functions on $ \mathbb{T} $ as a preparation for error analysis. In Section \ref{sec:main}, we provide error estimates. In Section \ref{sec:proof}, we give the proofs of the main theorems in Section \ref{sec:main}. Numerical results are presented in Section \ref{sec:test}. 

\section{Reformulation of the fully semi-Lagrangian scheme} \label{sec:conti} 
 
\subsection{Derivation of the scheme} 
 
We introduce some notation. We denote the norms of Sobolev spaces $ L^p (\mathbb{T}) $ and $ W^{s, p} (\mathbb{T}) $ by $ \norm{0, p}{\, \cdot \,} $ and $ \norm{s, p}{\, \cdot \,} $, respectively. We also use the Sobolev seminorm $ \snorm{s, p}{\, \cdot \,} = \norm{0, p}{\partial_x^s (\, \cdot \,)} $. For a positive integer $ s $, we define 
\begin{align} 
  \norm{s, 2, \ast}{v} = \left(\norm{0,2}{v}^2 + \snorm{s,2}{v}^2\right)^{1/2}, 
\end{align} 
which is equivalent to the standard $ H^s $-norm. We further define the norm $ \norm{s, 2, \Delta}{\,\cdot\,} $, which depends on the discretization parameter $ \Delta = (h, \Delta t) $, by 
\begin{align} 
  \norm{s, 2, \Delta}{v} = \left(\norm{0, 2}{v}^2 + \frac{h^{2s}}{\Delta t} \snorm{s, 2}{v}^2 \right)^{1/2}, \label{nsd} 
\end{align} 
and refer to it as the weighted $ H^s $-norm. We denote the Bochner space $W^{k,p}(0,T; H^s(\mathbb{T}))$ by  $W^{k,p} (H^s)$. 

In this paper, $ C $ denotes generic positive constants, independent of $ h $ and $ \Delta t $, whose value may change from line to line. We say that a positive quantity $ Q $ depends polynomially on $ A_1,\ldots,A_l $ when there exists a polynomial $ \Phi(A_1,\ldots,A_l) $ such that 
\begin{align} 
  Q \le \Phi(A_1,\ldots,A_l). 
\end{align} 
We use $ C_{\mathrm{P}} (A_1,\ldots,A_l) $ to denote a generic constant which depends polynomially on $ A_1,\ldots,A_l $. 

Let us briefly recall the derivation of the fully SL scheme \eqref{slj1}--\eqref{slj2}. Let a function $ w: \mathbb{T} \times [0, T] \ni (x,t) \mapsto w (x,t) \in \mathbb{R} $ be Lipschitz continuous with respect to $ x $. For $ (x, t) \in \mathbb{T} \times [0, T] $, let us define the trajectory $ X [w]: (0, T) \to \mathbb{T} $ by 
\begin{empheq}[left={\empheqlbrace}]{alignat=2} 
  & \frac{\mathrm{d}}{\mathrm{d} s} X [w] (s; x, t) = w (X [w] (s; x, t), s) \quad & \text{in} \quad \mathbb{T} \times (0, T), \\ 
  & X [w] (t; x, t) = x. 
\end{empheq} 
Let $ u $ be a smooth solution to \eqref{ivp1}--\eqref{ivp2}, and let $ x \in \mathbb{T} $ be fixed. 
For any $ s \in [0, T] $ and a sufficiently smooth function $ v: \mathbb{T} \times (0, T) \to \mathbb{R} $, we have 
\begin{align} 
  \frac{\mathrm{d}}{\mathrm{d} s} u (X_s, s) = \partial_t u (X_s, s) + f (u (X_s, s)) \partial_x u (X_s, s), \label{md} 
\end{align} 
where $ X_s = X [f \circ u] (s;  x, t_n) $. We denote $ u (\, \cdot \, , t_n) $ by $ u^n $. Integrating \eqref{md} with respect to $ s $, we have 
\begin{align} 
  u^n (x) & = u^{n-1} (X [f \circ u] (t_{n-1}; x, t_n)) \\ & \quad + \int_{t_{n-1}}^{t_n} \left[\partial_t u (X_s, s) + f (u (X_s, s)) \partial_x u (X_s, s)\right] \mathrm{d} s. 
\end{align} 
Using \eqref{ivp1}, we have 
\begin{align} 
  u^n (x) & = u^{n-1} (X [f \circ u] (t_{n-1}; x, t_n)) - \nu \int_{t_{n-1}}^{t_n} \partial_x^3 u (X [f \circ u] (s; x, t_n), s) \mathrm{d} s. \label{de1} 
\end{align} 

We now approximate the second term on the right-hand side of \eqref{de1}. By \eqref{g4}, $ \Lambda_4 $ satisfies 
\begin{align} 
  \frac{1}{k!} \sum_{(\gamma, \lambda) \in \Lambda_4} \gamma \lambda^k = \begin{cases} 
    1 & (k = 0), \\
    0 & (k = 1, 2), \\ 
    - 1 & (k = 3), \label{g4s} 
  \end{cases} 
\end{align} 
where we use $ 0! = 1 $. Similarly, \eqref{g5} implies 
\begin{align} 
  \frac{1}{k!} \sum_{(\gamma, \lambda) \in \Lambda_5} \gamma \lambda^k = \begin{cases} 
    1 & (k = 0), \\ 
    0 & (k = 1, 2, 4), \\ 
    - 1 & (k = 3). \label{g5s} 
  \end{cases}
\end{align} 
By \eqref{g4s}, \eqref{g5s} and Taylor's theorem, we have 
\begin{align} 
  \sum_{(\gamma, \lambda) \in \Lambda} \gamma v (x + \lambda (\nu \Delta t)^{1/3}) & = \sum_{k = 0}^{4} \sum_{(\gamma, \lambda) \in \Lambda} \frac{1}{k !} \gamma \lambda^k (\nu \Delta t)^{k / 3} \partial_x^k v (x) + O ((\nu \Delta t)^{5/3}) \\ 
  & = \begin{cases} 
    v (x) - \nu \Delta t \partial_x^3 v (x) + O ((\nu \Delta t)^{4/3}) &  (\Lambda = \Lambda_4), \\
    v (x) - \nu \Delta t \partial_x^3 v (x) + O ((\nu \Delta t)^{5/3}) & (\Lambda = \Lambda_5). 
  \end{cases} 
\end{align} 
Thus, we have 
\begin{align} 
  & u^{n-1} (X [f \circ u] (t_{n-1}; x, t_n)) - \nu \Delta t \partial_x^3 u^{n-1} (X [f \circ u] (t_{n-1}; x, t_n)) \\
  & = \sum_{(\gamma, \lambda) \in \Lambda} \gamma u^{n-1} (X [f \circ u] (t_{n-1}; x, t_n) + \lambda (\nu \Delta t)^{1/3}) + O ((\nu \Delta t)^{1 + r}), \label{de3} 
\end{align} 
where 
\begin{align}
  r = \begin{cases}
    \frac{1}{3} & (\Lambda = \Lambda_4), \\ 
    \frac{2}{3} & (\Lambda = \Lambda_5). 
  \end{cases}
\end{align}
We also use the approximation 
\begin{align} 
  \int_{t_{n-1}}^{t_n} \partial_x^3 u (X [f \circ u] (s; x, t_n), s) \mathrm{d} s = \Delta t \partial_x^3 u^{n - 1} (X [f \circ u] (t_{n-1}; x, t_n)) + O (\Delta t^2). \label{de2} 
\end{align} 
Applying \eqref{de3} and \eqref{de2} to \eqref{de1}, we obtain 
\begin{align} 
  u^n (x) & = \sum_{(\gamma, \lambda) \in \Lambda} \gamma u^{n-1} (X [f \circ u] (t_{n-1}; x, t_n) + \lambda (\nu \Delta t)^{1/3}) + O (\Delta t^{1 + r}). \label{de4} 
\end{align} 
Furthermore, by the backward Euler method, we have the approximation 
\begin{align} 
  X [f \circ u] (t_{n-1}; x, t_n) \approx x - f (u^n (x)) \Delta t. \label{de5} 
\end{align} 
Applying \eqref{de5} to \eqref{de4}, and then replacing $ u^n $ by $ \bar{u}^n $, we obtain the semi-discrete scheme in time 
\begin{align}
  \bar{u}^n (x) = \sum_{(\gamma, \lambda) \in \Lambda} \gamma \bar{u}^{n-1} (x - f (\bar{u}^n (x)) \Delta t + \lambda (\nu \Delta t)^{1/3}). \label{sd} 
\end{align} 
By introducing the spatial interpolation $ \tilde{\mathcal{I}}_h $, we obtain the fully discrete scheme \eqref{slj1}--\eqref{slj2}. This completes the derivation of the scheme. 

\subsection{Reformulation of the scheme} 

The scheme \eqref{slj1}--\eqref{slj2} is written in terms of the numerical solution at the discrete grid points $ (x_j, t_n) $. In this subsection, we rewrite this scheme in terms of the numerical solutions $ u_h^n $ defined on the domain $ \mathbb{T} $. For any $ v \in C (\mathbb{T}) $, we define $ X_{\Delta t}^1 [v]: \mathbb{T} \to \mathbb{T} $ by 
\begin{align} 
  X_{\Delta t}^1 [v] (x) = x - v (x) \Delta t. 
\end{align} 
The function $ X_{\Delta t}^1 [f \circ u^n] : \mathbb{T} \to \mathbb{T} $ is the first-order approximation of $ X [f \circ u] (t_{n-1}; \,\cdot\,, t_n) $ obtained by the backward Euler method. Let $ s \geq 2 $, $ \Delta t \in (0, 1) $, $ f \in W^{s, \infty} (\mathbb{T}) $ and $ D_{f,\Delta t} = \{v \in H^s (\mathbb{T}) \,|\, \snorm{1, \infty}{v} < (3 \snorm{1, \infty}{f}\Delta t)^{-1} \} $. For $ v \in D_{f,\Delta t} $, there exists a unique function $ w \in H^s (\mathbb{T}) $ such that 
\begin{align} 
  w & = v \circ {X_{\Delta t}^1} [f \circ w]. \label{sa} 
\end{align} 
We define the operator $\mathcal{S}_{f, \Delta t}^{\mathrm{A}} : D_{f,\Delta t} \to H^{s}(\mathbb{T})$ by $ \mathcal{S}_{f, \Delta t}^{\mathrm{A}}(v) = w $ with \eqref{sa}. The existence and uniqueness of $ w $ satisfying \eqref{sa} are ensured by the following lemma. 

\begin{lemma}[Lemmas 2.1 and A.1 of \cite{25Ta}] \label{lem:s} 
  Let $ s \geq 2 $ be an integer. Assume that $ v \in H^s (\mathbb{T}) $ and $ f \in W^{s, \infty} (\mathbb{T}) $. Suppose that 
  \begin{align} 
    3 \Delta t \snorm{1, \infty}{f} \snorm{1, \infty}{v} \leq 1 
  \end{align} 
  is satisfied. Then the following three statements hold: 
  \begin{enumerate}[(i)] 
    \item There exists a unique function $ w \in H^s (\mathbb{T}) $ such that \label{sa1} 
    \begin{align} 
      w = v \circ {X_{\Delta t}^1}[f \circ w]. \label{a1} 
    \end{align} 
    In the remainder of this lemma, we denote $ w $ by $ \mathcal{S}_{f, \Delta t}^{\mathrm{A}} (v) $. 
    \item The mapping $ X_{\Delta t}^1[f \circ \mathcal{S}_{f, \Delta t}^{\mathrm{A}} (v)]: \mathbb{T} \to \mathbb{T} $ is a bijection. 
    \item There exists a positive constant $ C_{\mathrm{P}}^{(1)} $ which depends polynomially only on $ \norm{s, \infty}{f} $ and $ \norm{s,2,\ast}{v} $ such that 
    \begin{align} 
      \norm{s,2,\ast}{\mathcal{S}_{f, \Delta t}^{\mathrm{A}} (v)} \leq \left(1 + C_{\mathrm{P}}^{(1)} \Delta t \right) \norm{s,2,\ast} {v}. \label{a2} 
    \end{align} 
  \end{enumerate} 
\end{lemma} 

For $ \Lambda = \Lambda_4 $ or $ \Lambda_5 $, let $ \mathcal{S}_{\Lambda, \Delta t}^{\mathrm{D}}: L^2 (\mathbb{T}) \to L^2 (\mathbb{T}) $ be defined by 
\begin{align} 
  (\mathcal{S}_{\Lambda, \Delta t}^{\mathrm{D}} v) (x) = \sum_{(\gamma, \lambda) \in \Lambda} \gamma v (x + \lambda (\nu \Delta t)^{1/3}). 
\end{align} 

In the following, we reformulate the scheme \eqref{slj1}--\eqref{slj2} using the operators $ \mathcal{S}_{f, \Delta t}^{\mathrm{A}} $ and $ \mathcal{S}_{\Lambda, \Delta t}^{\mathrm{D}} $. Let $ u_h^n = \tilde{\mathcal{I}}_h [U^n] $. For $ n \geq 1 $, we have 
\begin{align} 
  (\mathcal{S}_{\Lambda, \Delta t}^{\mathrm{D}} u_h^{n-1}) (x) = \sum_{(\gamma, \lambda) \in \Lambda} \gamma u_h^{n - 1} (x + \lambda (\nu \Delta t)^{1/3}), 
\end{align} 
and 
\begin{align} 
  & \mathcal{S}_{f, \Delta t}^{\mathrm{A}} (\mathcal{S}_{\Lambda, \Delta t}^{\mathrm{D}} u_h^{n - 1}) (x) \\ 
  & = \sum_{(\gamma, \lambda) \in \Lambda} \gamma u_h^{n - 1} (x - f (\mathcal{S}_{f, \Delta t}^{\mathrm{A}} (\mathcal{S}_{\Lambda, \Delta t}^{\mathrm{D}} u_h^{n - 1}) (x)) \Delta t + \lambda (\nu \Delta t)^{1/3}). \label{sl4} 
\end{align} 
Suppose that \eqref{slj2} admits a unique solution $ u_j^n $ for each $ (j, n) \in \{1, \ldots, N_x\} \times \{1, \ldots, N_t\} $. Comparing \eqref{slj2} and \eqref{sl4}, we have 
\begin{align} 
  u_j^n = \mathcal{S}_{f, \Delta t}^{\mathrm{A}} (\mathcal{S}_{\Lambda, \Delta t}^{\mathrm{D}} u_h^{n - 1}) (x_j). 
\end{align} 
Therefore, we have 
\begin{align}
  u_h^n = \tilde{\mathcal{I}}_h \left[\Big(\mathcal{S}_{f, \Delta t}^{\mathrm{A}} (\mathcal{S}_{\Lambda, \Delta t}^{\mathrm{D}} u_h^{n - 1}) (x_j)\Big)_{j=1}^{N_x}\right]. \label{sl8} 
\end{align}
Let the interpolation operator $ \mathcal{I}_h: C (\mathbb{T}) \to C (\mathbb{T}) $ be defined by 
\begin{align} 
  \mathcal{I}_h v = \tilde{\mathcal{I}}_h [(v (x_1), v (x_2), \ldots, v (x_{N_x}))]. 
\end{align} 
Then, the fully SL scheme \eqref{slj1}--\eqref{slj2} can be rewritten as 
\begin{empheq}[left={\empheqlbrace}]{alignat=2} 
  & u_h^0 = \mathcal{I}_h u_0, \label{sl1} \\ 
  & u_h^n = \mathcal{I}_h \left(\mathcal{S}_{f, \Delta t}^{\mathrm{A}} (\mathcal{S}_{\Lambda, \Delta t}^{\mathrm{D}} u_h^{n - 1}) \right). \label{sl2} 
\end{empheq} 

By Lemma \ref{lem:s}, the function $ u_h^n $ is uniquely determined by \eqref{sl2} provided that 
\begin{align} 
  \Delta t \snorm{1,\infty}{f}\snorm{1,\infty}{\mathcal{S}_{\Lambda, \Delta t}^{\mathrm{D}} u_h^n} \leq \frac{1}{3}. \label{sl9} 
\end{align} 
As we will see in the proof of Theorem \ref{thm:s} later, if the initial value problem \eqref{ivp1}--\eqref{ivp2} admits a sufficiently smooth solution and $ \Delta t $ is sufficiently small, $ \norm{s,2}{u_h^n} $ ($ s \geq 2 $) is bounded by a constant independent of $ h $, $ \Delta t $ and $ n $. Under this condition, we obtain the bound 
\begin{align} 
  \snorm{1, \infty}{u_h^n} \leq C, 
\end{align} 
where $ C $ is independent of $ h $, $ \Delta t $ and $ n $. Therefore, the condition \eqref{sl9} is satisfied for all $ n $ if $ \Delta t $ is sufficiently small. 

We introduce two choices for the interpolation operator $ \mathcal{I}_h $ in the scheme \eqref{sl1}--\eqref{sl2}. For positive integer $ l $, let $ \mathbb{P}^l (\lparen x_{j-1}, x_j\rbrack) $ be the set of all polynomials of degree at most $ l $ on the interval $ \lparen x_{j-1}, x_j\rbrack $, and the space of piecewise polynomials $ \mathbb{P}_h^l (\mathbb{T}) $ be defined by 
\begin{align} 
  \mathbb{P}_h^l (\mathbb{T}) = \left\{v \in L^2 (\mathbb{T}) \relmiddle v|_{\lparen x_{j-1}, x_j \rbrack} \in \mathbb{P}^l (\lparen x_{j-1}, x_j \rbrack) \text{ for all } j = 1, \ldots, N_x \right\}. 
\end{align} 
For $ s \geq 2 $, the spline interpolation operator $ \mathcal{I}_{h, \mathrm{spl}}^{2s-1}: C (\mathbb{T}) \to \mathbb{P}_h^{2 s - 1} (\mathbb{T}) \cap C^{2 s - 2} (\mathbb{T}) $ is defined by 
\begin{gather} 
  (\mathcal{I}_{h, \mathrm{spl}}^{2s-1} v) (x_j) = v (x_j), \text{ for all } j \in \{1, \ldots, N_x\}. \label{sp} 
\end{gather} 
In this paper,  when we use the spline interpolation, we suppose that $ \mathcal{I}_{h, \mathrm{spl}}^{2s-1} v $ is uniquely defined by \eqref{sp} for any $ v \in C (\mathbb{T}) $. The linear and cubic spline interpolations used in \cite{25Fe} are identical to $ \mathcal{I}_{h, \mathrm{spl}}^{2s-1} $ with $ s = 1 $ and $ s = 2 $, respectively. 

Let us introduce another interpolation. We define the Hermite interpolation operator $ \mathcal{I}_{h, \mathrm{Her}}^{2s-1}: C^{s - 1} (\mathbb{T}) \to \mathbb{P}_h^{2 s - 1} (\mathbb{T}) \cap C^{s - 1} (\mathbb{T}) $ by 
\begin{gather} 
  \partial_x^k (\mathcal{I}_{h, \mathrm{Her}}^{2s-1} v) (x_j) = \partial_x^k v (x_j), \text{ for all } (j, k) \in \{1, \ldots, N_x\} \times \{0, \ldots, s - 1\}. \label{her} 
\end{gather} 
If Hermite interpolation is used in the scheme \eqref{sl1}--\eqref{sl2}, the resulting scheme cannot be represented within the framework of \eqref{slj1}--\eqref{slj2}, since we need the numerical approximation of derivatives $ \partial_x^k u (x_j, t_n) $ ($ k = 1,\ldots,s-1 $). 
As a simple example, we exhibit the scheme with the cubic Hermite interpolation operator. 
Denote the numerical approximation of $ \partial_x u (x_j, t_n) $ by $ v_j^n $, and let $ \tilde{U}^{n} = (u_1^n, \ldots, u_{N_x}^n, v_1^n, \ldots, v_{N_x}^n) $. 
Let us define the cubic interpolation operator $ \tilde{\mathcal{I}}_{h, \mathrm{Her}}^{3}: \mathbb{R}^{2 N_x} \to \mathbb{P}_h^{3} (\mathbb{T}) \cap C^{1} (\mathbb{T}) $ be defined by 
\begin{align}
  \tilde{\mathcal{I}}_{h, \mathrm{Her}}^{3} [A] (x_j) = a_j, \quad 
  \partial_x \left(\tilde{\mathcal{I}}_{h, \mathrm{Her}}^{3} [A]\right) (x_j) = b_j
\end{align} 
for $ A = (a_1, \ldots, a_{N_x}, b_1, \ldots, b_{N_x}) $. Then, the scheme can be written as 
\begin{empheq}[left={\empheqlbrace}]{alignat=2} 
  u_j^0 & = u_0 (x_j), \quad v_j^0 = \partial_x u_0 (x_j), \\ 
  u_j^n & = \sum_{(\gamma, \lambda) \in \Lambda} \gamma \tilde{\mathcal{I}}_{h, \mathrm{Her}}^{3} [\tilde{U}^{n - 1}] (x_j - f (u_j^n) + \lambda (\nu \Delta t)^{1/3}), \label{slj1h} \\
  v_j^n & = \frac{w_j^n}{1 + w_j^n f^\prime (u_j^n) \Delta t}, \label{slj2h} 
\end{empheq} 
where 
\begin{align} 
  w_j^{n-1} = \sum_{(\gamma, \lambda) \in \Lambda} \gamma \partial_x (\tilde{\mathcal{I}}_{h, \mathrm{Her}}^{3} [\tilde{U}^{n - 1}]) (x_j - f (u_j^n) + \lambda (\nu \Delta t)^{1/3} ). 
\end{align} 
The numerical experiments in Section \ref{sec:test} will be performed using this scheme. 

\section{Main results}\label{sec:main} 

We prepare some assumptions for the interpolation operators. 
\begin{enumerate}[(P1)] 
  \item The operator $ \mathcal{I}_{h} $ satisfies \label{p:o} 
  \begin{align} 
    \int_\mathbb{T} \partial_x^s (v (x) - \mathcal{I}_{h} v (x)) \cdot  \partial_x^s \left(\mathcal{I}_{h} w\right) (x) \, \mathrm{d}x = 0. \label{po1} 
  \end{align} 
  for any $ v, w \in H^s (\mathbb{T}) $. Equality \eqref{po1} is equivalent to 
  \begin{align} 
    \snorm{s, 2}{(I - \mathcal{I}_{h}) v - \mathcal{I}_{h} w}^2 & = \snorm{s, 2}{(I - \mathcal{I}_{h}) v}^2 + \snorm{s, 2}{\mathcal{I}_{h} w}^2, \label{po2} 
  \end{align}
  where we denote the identity operator by $ I $. 
  \item There exists a positive constant $ C_s $ such that for any $ v \in H^s (\mathbb{T}) $, \label{p:s} 
  \begin{align} 
    \norm{0, 2}{(I - \mathcal{I}_{h}) v} & \leq C_s h^s \snorm{s, 2}{v}. \label{ps} 
  \end{align} 
  \item For an integer $ q > s $, there exists a positive constant $ C_q $ such that for any $ v \in H^q (\mathbb{T}) $, \label{p:q} 
  \begin{align} 
    \snorm{s, 2}{(I - \mathcal{I}_{h}) v} & \leq C_q h^{q - s} \snorm{q, 2}{v}. \label{pq} 
  \end{align} 
\end{enumerate} 

It is shown in Theorems 4, 7 and 9 in~\cite{SV67} that the spline interpolation $ \mathcal{I}_{h, \mathrm{spl}}^{2s-1} $ and the Hermite interpolation $ \mathcal{I}_{h, \mathrm{Her}}^{2s-1} $ satisfy (P\ref{p:o}), (P\ref{p:s}) and (P\ref{p:q}) with $ q = 2 s $. 

For $ n \in \{1, \dots, N_t\} $, we define the consistency error $ \tau_{\Delta t}^n $ of the semi-discrete scheme \eqref{sd} by 
\begin{align} 
  \tau_{\Delta t}^n & = \frac{1}{\Delta t} \Big( u^n - (\mathcal{S}_{\Lambda, \Delta t}^{\mathrm{D}} u^{n - 1}) \circ {X_{\Delta t}^1} [f \circ u^n]\Big). \label{tau} 
\end{align} 
Here, we also introduce assumptions for the parameter set $ \Lambda $. 
\begin{enumerate}[(D1)] 
  \item \label{ad1} The operator $ \mathcal{S}_{\Lambda, \Delta t}^{\mathrm{D}} $ satisfies 
  \begin{align} 
    \norm{0, 2}{\mathcal{S}_{\Lambda, \Delta t}^{\mathrm{D}} v} & \leq \norm{0, 2}{v}, \quad v \in L^2 (\mathbb{T}), \\ 
    \norm{s, 2, \ast}{\mathcal{S}_{\Lambda, \Delta t}^{\mathrm{D}} v} & \leq \norm{s, 2, \ast}{v}, \quad v \in H^s (\mathbb{T}), \\ 
    \norm{s, 2, \Delta}{\mathcal{S}_{\Lambda, \Delta t}^{\mathrm{D}} v} & \leq \norm{s, 2, \Delta}{v}, \quad v \in H^s (\mathbb{T}). 
  \end{align} 
  \item \label{ad2} There exists a positive constant $ C $ such that 
  \begin{align} 
    \norm{s + 1, \infty}{\mathcal{S}_{\Lambda, \Delta t}^{\mathrm{D}} v} \leq C \norm{s + 1, \infty}{v} \label{ad2i} 
  \end{align} 
  for any $ v \in W^{s + 1, \infty} (\mathbb{T}) $. 
  \item \label{ad3} There exist constants $ r \in \lparen 0,1 \rbrack $, independent of $ u $, and $ K > 0 $ which may depend on $ u $, such that 
  \begin{align} 
    \norm{s,2,\ast}{\tau_{\Delta t}^n} \leq K (u) \Delta t^r \label{ad3i} 
  \end{align} 
  for all $ n \in \{1, \ldots, N_t\} $, where $ u $ is the exact solution to \eqref{ivp1}--\eqref{ivp2}, and $ \tau_{\Delta t}^n $ is defined by \eqref{tau}. 
\end{enumerate} 

We will see that $ \Lambda_4 $ and $ \Lambda_5 $ satisfy (D\ref{ad1}) in Lemma \ref{lem:sd}. If $\Lambda$ is a finite set, then it is obvious that $\Lambda$ satisfies (D\ref{ad2}). In this paper, we do not consider the case where $\Lambda$ is infinite. The sets $ \Lambda_4 $ and $ \Lambda_5 $ satisfy (D\ref{ad3}) for $ r = 1 / 3 $ and $ r = 2 / 3 $, respectively, as we see later in \eqref{g4s} and \eqref{g5s}. 

In the following two theorem, we state the main results of this paper. 

\begin{theorem} \label{thm:s} 
  Assume the following hypotheses: 
  \begin{enumerate}[(i)] 
    \item \label{h1} $ f \in W^{s + 1, \infty} (\mathbb{R}) $. 
    \item \label{h2} The exact solution $ u $ to \eqref{ivp1}--\eqref{ivp2} belongs to $ W^{1,\infty}(H^{s + 2}) \cap L^{\infty} (H^{s + 4}) \cap L^{\infty}(H^{q}) $. 
    \item \label{h3} The interpolation operator $ \mathcal{I}_{h} $ satisfies the properties (P\ref{p:o}), (P\ref{p:s}) and (P\ref{p:q}) for positive integers $ s $ and $ q $ ($ s < q $). 
    \item \label{h4} The parameter set $ \Lambda $ satisfies (D\ref{ad1}), (D\ref{ad2}) and (D\ref{ad3}) for some $ r \in \lparen 0, 1 \rbrack $. 
  \end{enumerate} 
  Let $ C_\Delta $ and $ \varepsilon $ be positive constants. 
  Then there exist $ (h_0, \Delta t_0) \in (0, 1)^2 $ and $ C_{\mathrm{E}} > 0 $ depending on $ C_\Delta $, $ \varepsilon $, $ \norm{L^\infty (W^{s+1,\infty})}{u} $, $ \norm{L^{\infty}(H^{q})}{u} $, $ \norm{s + 1, \infty}{f} $ and $ K(u) $ such that, if $ \Delta = (h, \Delta t) \in (0, h_0) \times (0, \Delta t_0) $ satisfies 
  \begin{align} 
    h^{2s} / \Delta t \leq C_\Delta \quad \text{and} \quad h^{2(q - s)} / \Delta t^{1 + \varepsilon} \leq C_\Delta, \label{hyp:d} 
  \end{align} 
  the following estimate holds: 
  \begin{align} 
    \norm{s, 2, \ast}{u^n - u_h^n} \leq C_{\mathrm{E}} \left(\Delta t^{r} + \frac{h^{q - s}}{\Delta t^{1/2}}\right), \quad n = 0, \ldots, N_t. \label{b30} 
  \end{align} 
\end{theorem} 

\begin{theorem} \label{thm:2} 
  Assume the hypotheses (\ref{h1})--(\ref{h4}) in Theorem \ref{thm:s}. Let $ C_\Delta $ and $ \varepsilon $ be positive constants. Then there exist $ (h_0, \Delta t_0^\prime) \in (0, 1)^2 $ and $  C_{\mathrm{E}}^\prime > 0 $, depending on $ C_\Delta $, $ \varepsilon $, $ \norm{L^\infty (W^{s+1,\infty})}{u} $, $ \norm{L^{\infty}(H^{q})}{u} $, $ \norm{s + 1, \infty}{f} $ and $ K(u) $, such that, if $ \Delta = (h, \Delta t) \in (0, h_0) \times (0, \Delta t_0) $ satisfies the condition \eqref{hyp:d}, the following estimate holds: 
  \begin{align} 
    \norm{s, 2, \Delta}{u^n - u_h^n} \leq C_{\mathrm{E}}^\prime \left(\Delta t^r + \frac{h^q}{\Delta t}\right), \quad n = 0, \ldots, N_t. \label{b26} 
  \end{align} 
  In particular, \eqref{b26} implies the $ L^2 $-error estimate 
  \begin{align} 
    \norm{0,2}{u^n - u_h^n} \leq C_{\mathrm{E}}^\prime \left(\Delta t^r + \frac{h^q}{\Delta t}\right), \quad n = 0, \ldots, N_t. \label{b26-2} 
  \end{align} 
\end{theorem} 

For instance, if we use the cubic spline or Hermite interpolation operators, the estimates \eqref{b30} and \eqref{b26-2} can be written as 
\begin{align} 
  \norm{s, 2, \ast}{u^n - u_h^n} \leq C_{\mathrm{E}} \left(\Delta t^{r} + \frac{h^2}{\Delta t^{1/2}}\right), 
\end{align} 
and 
\begin{align} 
  \norm{0,2}{u^n - u_h^n} \leq C_{\mathrm{E}}^\prime \left(\Delta t^r + \frac{h^4}{\Delta t}\right). 
\end{align} 

\section{Key lemmas and the proofs of the main theorems}\label{sec:proof} 

\subsection{Preparations} 

The following lemmas can be found in \cite{25Ta}. 

\begin{lemma}[Lemma~A.3 of \cite{25Ta}]\label{lem:c} 
  Let $ v, w \in H^s (\mathbb{T}) $ and $ g \in W^{s,\infty} (\mathbb{T}) $ for $ s \geq 2 $. 
  Suppose that $ w $ satisfies $ \snorm{1, \infty}{w} \leq 1 / (2 \Delta t \snorm{1, \infty}{g}) $. 
  Then there exists a positive constant $ C_{\mathrm{P}}^{(3)} $ which depends polynomially on $ \norm{s, \infty}{g} $ and $ \norm{s, 2, \ast}{w} $ such that 
  \begin{align} 
    \norm{s, 2, \ast}{v \circ {X_{\Delta t}^1}[g \circ w]} \leq (1 + C_{\mathrm{P}}^{(3)} \Delta t) \norm{s, 2, \ast}{v}. \label{c2} 
  \end{align} 
  Furthermore, suppose there exists a positive constant $ C_\Delta $ such that $ h^{2s} / \Delta t \leq C_\Delta $. Then there exists a positive constant $ C_{\mathrm{P}}^{(4)} $ which depends polynomially on $ \norm{s, \infty}{g} $ and $ \norm{s, 2, \ast}{w} $ such that 
  \begin{align} 
    \norm{s, 2, \Delta}{v \circ {X_{\Delta t}^1}[g \circ w]} \leq (1 + C_{\mathrm{P}}^{(4)} \Delta t) \norm{s, 2, \Delta}{v}. \label{c2-2} 
  \end{align} 
\end{lemma} 

\begin{lemma}[Lemma~A.4 of \cite{25Ta}] \label{lem:d} 
  Let $ w_1 $, $ w_2 \in H^s (\mathbb{T}) $, $ g $, $ v \in W^{s + 1, \infty} (\mathbb{T}) $ for an integer $ s \geq 2 $. Then there exists a positive constant $ C_{\mathrm{P}}^{(5)} $ which depends polynomially on $ \norm{s + 1, \infty}{g} $, $ \norm{s, 2, \ast}{w_1} $ and $ \norm{s, 2, \ast}{w_2} $ such that 
  \begin{align}
    & \norm{s, 2, \ast}{v \circ {X_{\Delta t}^1} [g \circ w_1] - v \circ {X_{\Delta t}^1} [g \circ w_2]} \\ & \leq C_{\mathrm{P}}^{(5)} \Delta t \norm{s + 1, \infty}{v} \norm{s, 2, \ast}{w_1 - w_2}. \label{d4} 
  \end{align} 

  Furthermore, suppose that there exists a positive constant $ C_\Delta $ such that $ h^{2 s} / \Delta t \leq C_\Delta $. Then there exists a positive constant $ C_{\mathrm{P}}^{(6)} $ which depends polynomially on $ \norm{s + 1, \infty}{g} $, $ \norm{s, 2, \ast}{w_1} $ and $ \norm{s, 2, \ast}{w_2} $ such that 
  \begin{align} 
    & \norm{s, 2, \Delta}{v \circ {X_{\Delta t}^1} [g \circ w_1] - v \circ {X_{\Delta t}^1} [g \circ w_2]} \\ & \leq C_{\mathrm{P}}^{(6)} \Delta t \norm{s + 1, \infty}{v} \norm{s, 2, \Delta}{w_1 - w_2}. \label{d22} 
  \end{align} 
\end{lemma} 

A similar inequality to \eqref{d4} appears in Lemma 4.7 of \cite{16NoTa}, where the $L^2$-norm is considered instead of the $H^s$-norm. 

The statement in the next lemma can be found in Lemmas 3.1 and 3.3 of \cite{25Ta}. Inequality \eqref{lem2} was already derived in Lemmas 3.6 and 4.4 of \cite{24KaTa} for the cubic spline and Hermite interpolation operators. 

\begin{lemma} 
  Suppose that the interpolation operator $ \mathcal{I}_{h} $ possesses the properties (P\ref{p:o}) and (P\ref{p:s}). Then there exists a positive constant $ C $ such that for any $ v \in H^s (\mathbb{T}) $ 
  \begin{align} 
    \norm{s, 2, \Delta}{\mathcal{I}_{h} v} & \leq \left(1 + C \Delta t \right) \norm{s, 2, \Delta}{v}. \label{lem2} 
  \end{align} 
  Furthermore, suppose that there exists a positive constant $ C_\Delta $ such that $ h^{2s} / \Delta t \leq C_\Delta $. Then there exists a positive constant $ C $ such that for any $ v \in H^s (\mathbb{T}) $ 
  \begin{align} 
    \norm{s, 2, *}{\mathcal{I}_{h} v} & \leq \left(1 + C \Delta t \right) \norm{s, 2, *}{v}. \label{lemw} 
  \end{align} 
\end{lemma} 

\subsection{Consistency} 
In this subsection, we establish estimates for $ \tau_{\Delta t}^n $ defined in \eqref{tau}. Let $ \delta = (\nu \Delta t)^{1/3} $. Then $ \tau_{\Delta t}^n $ can be expressed as 
\begin{align} 
  \tau_{\Delta t}^n (x) & = \frac{1}{\Delta t} \left[u^n (x) - \sum_{(\gamma, \lambda) \in \Lambda} \gamma u^{n - 1} (x - f (u^n (x)) \Delta t + \lambda \delta) \right]. 
\end{align} 
By Taylor's theorem, we have 
\begin{align} 
  & u^{n-1} (x - f (u^n (x)) \Delta t + \lambda \delta) \\ 
  & = u^{n-1} (x) + (- f (u^n (x)) \Delta t + \lambda \delta) \partial_x u^{n-1} (x) \\ 
  &\quad + \frac{1}{2} (- f (u^n (x)) \Delta t + \lambda \delta)^2 \partial_x^2 u^{n-1} (x) 
    + \frac{1}{6} (- f (u^n (x)) \Delta t + \lambda \delta)^3 \partial_x^3 u^{n-1} (x) \\ 
  & \quad + \frac{\Delta t^{4/3}}{6} \int_{0}^{- f (u^n(x)) \Delta t^{2/3} + \lambda \nu^{1/3}} y^3 \partial_x^4 u (x + \Delta t^{1/3} y, t_{n-1}) \, \mathrm{d}y. \label{tay} 
\end{align} 
Using \eqref{ivp1} and \eqref{tay} to \eqref{tau}, we have 
\begin{align} 
  \tau_{\Delta t}^n = \tau_{\Delta t,1}^n + \tau_{\Delta t,2}^n + \tau_{\Delta t,3}^n + \tau_{\Delta t,4}^n, 
\end{align} 
where 
\begin{align} 
  \tau_{\Delta t,1}^n (x) & = \frac{1}{\Delta t} \left(u^n(x) - u^{n-1}(x)\right) - \partial_t u (x, t_{n-1}), \\ 
  \tau_{\Delta t,2}^n (x) & = \left(f (u^n(x)) - f (u^{n-1}(x))\right) \partial_x u (x, t_{n-1}), \\ 
  \tau_{\Delta t,3}^n (x) & = - \frac{\Delta t}{2} f (u^n (x))^2 + \frac{\Delta t^2}{6} f (u^n (x))^3 \partial_x^3 u (x, t_{n-1}), \label{tu3} \\ 
  \tau_{\Delta t,4}^n (x) & = \sum_{(\gamma, \lambda) \in \Lambda} \tau_{\Delta t,4,\gamma,\lambda}^n (x), 
\end{align} 
and 
\begin{align} 
  \tau_{\Delta t,4,\gamma,\lambda} (x) = \frac{\gamma \Delta t^{1/3}}{6} \int_{0}^{- f (u^n(x)) \Delta t^{2/3} + \lambda \nu^{1/3}} y^3 \partial_x^4 u (x + \Delta t^{1/3} y, t_{n-1}) \, \mathrm{d}y. \label{tu4}
\end{align} 
By Taylor's theorem, we have 
\begin{align} 
  \tau_{\Delta t,1}^n (x) & = \Delta t \int_{0}^{1} \theta \partial_t^2 u (x, t_{n-1} + \theta \Delta t) \, \mathrm{d} \theta \label{tu1} 
\end{align} 
and 
\begin{align} 
  \tau_{\Delta t,2}^n (x) = \Delta t \partial_x u (x, t_{n-1}) \int_{0}^{1} f^\prime (u (x, t + \theta \Delta t)) \partial_t u (x, t + \theta \Delta t) \, \mathrm{d} \theta. \label{tu2} 
\end{align} 
From \eqref{tu3}, \eqref{tu4}, \eqref{tu1} and \eqref{tu2}, we deduce that there exists a positive constant $ C_{\tau}^{(1)} $, independent of $ \Delta t $, such that 
\begin{align} 
  \norm{0, 2}{\tau_{\Delta t,i}^n} & \leq C_{\tau}^{(1)} \Delta t \quad (i = 1,2,3), \\ 
  \norm{0, 2}{\tau_{\Delta t,4}^n} & \leq C_{\tau}^{(1)} \Delta t^{1/3}, \\ 
  \snorm{s, 2}{\tau_{\Delta t,i}^n} & \leq C_{\tau}^{(1)} \Delta t  \quad (i = 1,2,3), \\ 
  \snorm{s, 2}{\tau_{\Delta t,4}^n} & \leq C_{\tau}^{(1)} \Delta t^{1/3}, 
\end{align} 
where $ C_{\tau}^{(1)} $ depend on $ \nu $, $ \norm{W^{2,\infty} (t_{n-1}, t_n; H^s (\mathbb{T}))}{u} $, $ \norm{L^{\infty} (t_{n-1}, t_n; H^{s+4} (\mathbb{T}))}{u} $ and $ \norm{W^{{s+1},\infty} (\mathbb{R})}{f} $. 
Therefore, we obtain that there exists a constant $ C_{\tau}^{(2)} $, independent of $ \Delta t $, such that 
\begin{align} 
  \norm{s,2,\ast}{\tau_{\Delta t}^n} \leq C_{\tau}^{(2)} \Delta t^{1/3}. 
\end{align} 
This estimate implies that $ \Lambda_4 $ defined in \eqref{g4} satisfies the condition (D\ref{ad3}) with $ r = 1/3 $. 

In a similar manner, when $ \Lambda_5 $ is employed, we obtain that there exists a constant $ C_{\tau}^{(3)} $, independent of $ \Delta t $, such that 
\begin{align} 
  \norm{s,2,\ast}{\tau_{\Delta t}^n} \leq C_{\tau}^{(3)} \Delta t^{2/3}, 
\end{align} 
which implies that $ \Lambda_5 $ satisfies the condition (D\ref{ad3}) with $ r = 2 / 3 $. 

\subsection{Proofs of the main theorems} 

The following lemma concerns the hypothesis (D\ref{ad1}), which is related to the stability of the operator $ \mathcal{S}_{\Lambda,\Delta t}^{\mathrm{D}} $. 

\begin{lemma} \label{lem:sd} 
  Let $ s $ be a positive integer, and suppose that $ \Lambda $ is either $ \Lambda_4 $ or $ \Lambda_5 $. For any $ v \in H^s (\mathbb{T}) $, the following inequalities hold: 
  \begin{align} 
    \norm{s,2,\ast}{\mathcal{S}_{\Lambda, \Delta t}^{\mathrm{D}} v} \leq \norm{s,2,\ast}{v}, \label{sv4} 
  \end{align} 
  and 
  \begin{align} 
    \norm{s,2,\Delta}{\mathcal{S}_{\Lambda, \Delta t}^{\mathrm{D}} v} \leq \norm{s,2,\Delta}{v}. \label{sv5} 
  \end{align} 
\end{lemma} 

\begin{proof} 
  For any $ y \in \mathbb{R} $, we define the translation operator $ \mathcal{T}_{y}: L^2 (\mathbb{T}) \to L^2 (\mathbb{T}) $ by $ (\mathcal{T}_{y} v) (x) = v (x - y) $ ($ v \in L^2 (\mathbb{T}) $). 
  Then, the operator $ \mathcal{S}_{\Lambda, \Delta t}^{\mathrm{D}} $ can be written as 
  \begin{align} 
    \mathcal{S}_{\Lambda, \Delta t}^{\mathrm{D}} v = \sum_{(\gamma, \lambda) \in \Lambda} \gamma \mathcal{T}_{-\lambda \delta} v. 
  \end{align} 
  For any $ y_1 $, $ y_2 $, $ y_3 \in \mathbb{R} $, we have 
  \begin{align} 
  \left(\mathcal{T}_{y_1 + y_2} v, \mathcal{T}_{y_1 + y_3} v\right)_\mathbb{T} = \left(\mathcal{T}_{y_2} v, \mathcal{T}_{y_3} v\right)_\mathbb{T}, \label{tr} 
  \end{align} 
  where the $ L^2 $ inner product $ \left(\,\cdot\, , \,\cdot\,\right)_\mathbb{T} $ is defined by $ \left(v_1, v_2 \right)_\mathbb{T} = \int_\mathbb{T} v_1 (x) v_2 (x) \, \mathrm{d}x $. 

  First, we consider the case $ \Lambda = \Lambda_4 $. Using \eqref{tr}, we have 
  \begin{align} 
    & \left(\norm{0,2}{v}^2 - \norm{0,2}{\frac{1}{4} \left(\mathcal{T}_{y} v + v + 3 \mathcal{T}_{-y} v - \mathcal{T}_{-2y} v\right)}^2\right)_{\mathbb{T}} \\ 
    & = \frac{1}{4} \left(v, v \right)_\mathbb{T} - \frac{1}{8} \left(v, \mathcal{T}_{y} v \right)_\mathbb{T} - \frac{1}{8} \left(v, \mathcal{T}_{2y} v \right)_\mathbb{T} + \frac{1}{16} \left(v, \mathcal{T}_{3y} v \right)_\mathbb{T} \\ 
    & = \frac{1}{16} \norm{0,2}{\mathcal{T}_{y} v - v - \mathcal{T}_{-y} v + \mathcal{T}_{-2y} v}^2 \geq 0. \label{sv2} 
  \end{align} 
  Substituting $ y = \sqrt[3]{4} \delta $, we obtain 
  \begin{align} 
    \norm{0,2}{\mathcal{S}_{\Lambda, \Delta t}^{\mathrm{D}} v} \leq \norm{0,2}{v}. \label{sv1} 
  \end{align} 

  Let $ w \in H^s (\mathbb{T}) $. Substituting $ v = \partial_x^s w $ in \eqref{sv2}, we have 
  \begin{align} 
    \snorm{s,2}{\mathcal{S}_{\Lambda, \Delta t}^{\mathrm{D}} w} \leq \snorm{s,2}{w}. \label{sv3} 
  \end{align} 
  Combining \eqref{sv2} and \eqref{sv3}, we obtain \eqref{sv4} and \eqref{sv5}. 

  Next, we consider the case $ \Lambda = \Lambda_5 $. Using \eqref{tr}, we have 
  \begin{align} 
    & \norm{0,2}{v}^2 - \norm{0,2}{\frac{3}{16} \mathcal{T}_{2\delta} v + \frac{3}{8} v + \frac{3}{4} \mathcal{T}_{-2\delta} v - \frac{3}{8} \mathcal{T}_{-4\delta} v + \frac{1}{16} \mathcal{T}_{-6\delta} v}^2 \\ 
    & = \frac{3}{256} \left(5 (v, v)_{\mathbb{T}} - 4 (v, \mathcal{T}_{2\delta} v)_{\mathbb{T}} - 4 (v, \mathcal{T}_{4\delta} v)_{\mathbb{T}} + 4 (v, \mathcal{T}_{6\delta} v)_{\mathbb{T}} - (v, \mathcal{T}_{8\delta} v)_{\mathbb{T}}\right) \\ 
    & = \frac{3}{256} \norm{0,2}{v - 2 \mathcal{T}_{2\delta} v + 2 \mathcal{T}_{6\delta} v - \mathcal{T}_{8\delta} v}^2 \geq 0, 
  \end{align} 
  which implies 
  \begin{align} 
    \norm{0,2}{\mathcal{S}_{\Lambda, \Delta t}^{\mathrm{D}} v} \leq \norm{0,2}{v}. \label{sv6} 
  \end{align} 
  By substituting $ w = \partial_x^s v $ to \eqref{sv6}, we obtain 
  \begin{align} 
    \snorm{s,2}{\mathcal{S}_{\Lambda, \Delta t}^{\mathrm{D}} w} \leq \snorm{s,2}{w}. \label{sv7} 
  \end{align} 
  From \eqref{sv6} and \eqref{sv7}, we deduce \eqref{sv4} and \eqref{sv5} for the case $ \Lambda = \Lambda_5 $. 
\end{proof} 

In the following, we give the proof of Theorem \ref{thm:s}. The proof consists of three steps: 
\begin{enumerate}[Step 1.] 
  \item We first establish the estimate 
  \begin{align} 
    \norm{s, 2, \ast}{u^n - u_h^n}^2 & \leq \left(1 + C (u, u_h^{n-1}) \Delta t \right) \norm{s, 2, \ast}{u^{n-1} - u_h^{n-1}}^2 \\ & \quad + C (u, u_h^{n-1}) \Delta t^{1 + \bar{\varepsilon}}, \label{step1} 
  \end{align} 
  for sufficiently small $ \Delta t $, where $ C (u, u_h^{n-1}) $ and $ \bar{\varepsilon} $ are positive constants independent of $ \Delta t $. 
  \item Next, we derive the bound 
  \begin{align} 
    \norm{s, 2, \ast}{u_h^n} \leq 2 \sup_{t\in (0, T)} \norm{s, 2, \ast}{u (t)}, \label{step2}
  \end{align} 
  for sufficiently small $ h $ and $ \Delta t $. 
  \item Finally, combining \eqref{step1} and \eqref{step2}, we obtain \eqref{b30}. 
\end{enumerate} 

\begin{proof}[Proof of Theorem~\ref{thm:s}] 
  In this proof, we do not write the dependence on $ \norm{s + 1, \infty}{f} $ explicitly. We denote $ u^n - u_h^n $ by $ e_h^n $. 
  
  \paragraph{Step 1.} 
  First, we estimate $ e_h^0 = (I - \mathcal{I}_h) u_0 $. 
  By (P\ref{p:s}) and (P\ref{p:q}), we have 
  \begin{align} 
    \norm{s,2,\ast}{e_h^0} & = \left(\norm{0,2}{(I - \mathcal{I}_h) u_0}^2 + \snorm{s,2}{(I - \mathcal{I}_h) u_0}^2\right)^{1/2} \\ 
    & \leq C_0 h^{q - s}, \label{b0} 
  \end{align} 
  for some positive constant $ C_0 $ depending on $  \snorm{q,2}{u_0} $. 

  Next, we consider $ e_h^n $ for $ n \geq 1 $. 
  By \eqref{sl2}, we have 
  \begin{align} 
    e_h^n & = u^n - u_h^n \\ 
    & = u^n - \mathcal{I}_h \left(\mathcal{S}_{f, \Delta t}^{\mathrm{A}} (\mathcal{S}_{\Lambda, \Delta t}^{\mathrm{D}} u_h^{n-1})\right) \\ 
    & = \mathcal{I}_h \left(u^n - \mathcal{S}_{f, \Delta t}^{\mathrm{A}} ( \mathcal{S}_{\Lambda, \Delta t}^{\mathrm{D}} u_h^{n-1})\right) + (I - \mathcal{I}_h) u^n. \label{b6} 
  \end{align} 
  We define 
  \begin{align}
    \eta_h^n & = (\mathcal{S}_{\Lambda, \Delta t}^{\mathrm{D}} u^{n - 1}) \circ {X_{\Delta t}^1} [f \circ \mathcal{S}_{f, \Delta t}^{\mathrm{A}} (\mathcal{S}_{\Lambda, \Delta t}^{\mathrm{D}} u_h^{n - 1})] - \mathcal{S}_{f, \Delta t}^{\mathrm{A}} (\mathcal{S}_{\Lambda, \Delta t}^{\mathrm{D}} u_h^{n - 1}), \label{cs4} \\ 
    \theta_h^n & = (\mathcal{S}_{\Lambda, \Delta t}^{\mathrm{D}} u^{n - 1}) \circ {X_{\Delta t}^1} [f \circ u^n] \\ & \quad - (\mathcal{S}_{\Lambda, \Delta t}^{\mathrm{D}} u^{n - 1}) \circ {X_{\Delta t}^1} [f \circ \mathcal{S}_{f, \Delta t}^{\mathrm{A}} (\mathcal{S}_{\Lambda, \Delta t}^{\mathrm{D}} u_h^{n - 1})], \label{cs5}
  \end{align} 
  and 
  \begin{align} 
    \rho_h^n = (I - \mathcal{I}_h) u^n. 
  \end{align} 
  Recall that $ \tau_{\Delta t}^n $ is defined by \eqref{tau}. With these notations, we can rewrite \eqref{b6} as 
  \begin{align}
    e_h^n = \mathcal{I}_{h} (\eta_h^n + \theta_h^n + \Delta t \tau_{\Delta t}^n) + \rho_h^n. \label{b6-1}
  \end{align} 
  By \eqref{po2}, we have 
  \begin{align}
    \snorm{s, 2}{e_h^n}^2 & = \snorm{s, 2}{\mathcal{I}_{h} (\eta_h^n + \theta_h^n + \Delta t \tau_{\Delta t}^n)}^2 + \snorm{s, 2}{\rho_h^n}^2. \label{b6-2} 
  \end{align} 
  For any $ a_1, a_2, a_3 \in \mathbb{R} $ and $ \Delta t \in (0, 1) $, it holds that 
  \begin{align} 
    (a_1 + a_2 + a_3)^2 \leq (1 + \Delta t) a_1^2 + \frac{4}{\Delta t} (a_2^2 + a_3^2). \label{b6-3} 
  \end{align} 
  Applying \eqref{b6-3} to \eqref{b6-2}, we obtain 
  \begin{align} 
    \snorm{s, 2}{e_h^n}^2 & = \left(1 + \Delta t \right) \snorm{s, 2}{\mathcal{I}_{h} \eta_h^n}^2 + \frac{4}{\Delta t} \left(\snorm{s, 2}{\mathcal{I}_{h} \theta_h^n}^2 + \Delta t^2 \snorm{s, 2}{\mathcal{I}_{h} \tau_{\Delta t}^n}^2\right) + \snorm{s, 2}{\rho_h^n}^2. \label{b12-1} 
  \end{align} 
  Taking the $ L^2 $-norm of \eqref{b6-2}, we have 
  \begin{align} 
    \norm{0,2}{e_h^n}^2 = \norm{0,2}{\mathcal{I}_{h} (\eta_h^n + \theta_h^n + \Delta t \tau_{\Delta t}^n) + \rho_h^n}^2. \label{b6-4} 
  \end{align} 
  For any $ a_1, a_2, a_3, a_4 \in \mathbb{R} $ and $ \Delta t \in (0, 1) $, it holds that 
  \begin{align} 
    (a_1 + a_2 + a_3 + a_4)^2 \leq (1 + \Delta t) a_1^2 + \frac{6}{\Delta t} (a_2^2 + a_3^2 + a_4^2). \label{b6-5} 
  \end{align} 
  Applying \eqref{b6-5} to \eqref{b6-4}, we obtain 
  \begin{align} 
    \norm{0, 2}{e_h^n}^2 & \leq (1 + \Delta t) \norm{0, 2}{\mathcal{I}_{h} \eta_h^n}^2 \\ & \quad + \frac{6}{\Delta t} \left(\norm{0, 2}{\mathcal{I}_{h} \theta_h^n}^2 + \Delta t \norm{0, 2}{\mathcal{I}_{h} \tau_{\Delta t}^n}^2 + \norm{0, 2}{\mathcal{I}_{h} \rho_h^n}^2\right). \label{b12-2} 
  \end{align} 
  Combining \eqref{b12-1} and \eqref{b12-2}, we obtain 
  \begin{align} 
    \norm{s, 2, \ast}{e_h^n}^2
    & \leq (1 + \Delta t) \norm{s, 2, \ast}{\mathcal{I}_{h} \eta_h^n}^2 \\ & \quad + \frac{6}{\Delta t} \left(\norm{s, 2, \ast}{\mathcal{I}_{h} \theta_h^n}^2 + \Delta t^2 \norm{s, 2, \ast}{\mathcal{I}_{h} \tau_{\Delta t}^n}^2 + \norm{0, 2}{\rho_h^n}^2\right)+ \snorm{s, 2}{\rho_h^n}^2. 
  \end{align} 
  By \eqref{lem2} and $ \Delta t < 1 $, we have 
  \begin{align} 
    \norm{s, 2, \ast}{e_h^n}^2 & \leq (1 + C \Delta t) \norm{s, 2, \ast}{\eta_h^n}^2 + \frac{C}{\Delta t} \norm{s, 2, \ast}{\theta_h^n}^2 \\ 
    & \quad + C \Delta t \norm{s, 2, \ast}{\tau_{\Delta t}^n}^2 + \frac{C}{\Delta t} \norm{0, 2} {\rho_h^n}^2 + \snorm{s, 2}{\rho_h^n}^2. \label{b12} 
  \end{align} 

  Let us estimate $ \eta_h^n $. By the definitions of $ \eta_h^n $ and $ \mathcal{S}_{f, \Delta t}^{\mathrm{A}} $, we have 
  \begin{align} 
    \eta_h^n & = (\mathcal{S}_{\Lambda, \Delta t}^{\mathrm{D}} u^{n-1}) \circ {X_{\Delta t}^1} [f \circ \mathcal{S}_{f, \Delta t}^{\mathrm{A}} (\mathcal{S}_{\Lambda, \Delta t}^{\mathrm{D}} u_h^{n - 1})] \\ & \quad - (\mathcal{S}_{\Lambda, \Delta t}^{\mathrm{D}} u_h^{n-1}) \circ {X_{\Delta t}^1} [f \circ \mathcal{S}_{f, \Delta t}^{\mathrm{A}} (\mathcal{S}_{\Lambda, \Delta t}^{\mathrm{D}} u_h^{n - 1})] \\
    & = (\mathcal{S}_{\Lambda, \Delta t}^{\mathrm{D}} e_h^{n - 1}) \circ {X_{\Delta t}^1} [f \circ \mathcal{S}_{f, \Delta t}^{\mathrm{A}} (\mathcal{S}_{\Lambda, \Delta t}^{\mathrm{D}} u_h^{n - 1})]. \label{cs1} 
  \end{align} 
  Applying \eqref{c2} to \eqref{cs1}, we have 
  \begin{align} 
    \norm{s,2,\ast}{\eta_h^n} \leq \left(1 + C_{\mathrm{P}} (\norm{s,2,\ast}{\mathcal{S}_{f, \Delta t}^{\mathrm{A}} (\mathcal{S}_{\Lambda, \Delta t}^{\mathrm{D}} u_h^{n - 1})}) \Delta t \right) \norm{s,2,\ast}{\mathcal{S}_{\Lambda, \Delta t}^{\mathrm{D}} e_h^{n - 1}}. \label{cs3} 
  \end{align} 
  Let us define 
  \begin{align} 
    \Delta t_{h,1}^n =  
    \begin{cases}
      \min \left\{1, \frac{1}{3 \snorm{1,\infty}{f}\snorm{1,\infty}{u_h^{n-1}}} \right\} & \text{if} \quad \snorm{1,\infty}{f}\snorm{1,\infty}{u_h^{n-1}} > 0, \\
      1 & \text{if} \quad \snorm{1,\infty}{f}\snorm{1,\infty}{u_h^{n-1}} = 0. 
    \end{cases} 
  \end{align} 
  If $ \Delta t \leq \Delta t_{h,1}^n $, by \eqref{a2} and \eqref{sv4}, we have 
  \begin{align} 
    \norm{s, 2, \ast}{\mathcal{S}_{f, \Delta t}^{\mathrm{A}} (\mathcal{S}_{\Lambda, \Delta t}^{\mathrm{D}} u_h^{n - 1})} & \leq \left(1 + C_{\mathrm{P}} (\norm{s, 2, \ast}{\mathcal{S}_{\Lambda, \Delta t}^{\mathrm{D}} u_h^{n - 1}}) \Delta t \right) \norm{s, 2, \ast}{\mathcal{S}_{\Lambda, \Delta t}^{\mathrm{D}} u_h^{n - 1}} \\ 
    & \leq C_{\mathrm{P}} (\norm{s, 2, \ast}{u_h^{n - 1}}). \label{b8} 
  \end{align} 
  Applying \eqref{sv4} and \eqref{b8} to \eqref{cs3}, we have 
  \begin{align} 
    \norm{s,2,\ast}{\eta_h^n} \leq \left(1 + C_{\mathrm{P}} (\norm{s,2,\ast}{u_h^{n - 1}}) \Delta t \right) \norm{s,2,\ast}{e_h^{n - 1}}. \label{b2} 
  \end{align} 

  Let us estimate $ \rho_h^n $. By (P\ref{p:s}) and (P\ref{p:q}), we have 
  \begin{align} 
    \norm{0,2}{\rho_h^n} = \norm{0,2}{(I - \mathcal{I}_h) u^n} \leq C h^{q} \snorm{q,2}{u^n}, \label{b31-1} 
  \end{align} 
  and 
  \begin{align} 
    \snorm{s,2}{\rho_h^n} = \snorm{s, 2}{(I - \mathcal{I}_h) u^n} \leq C h^{q - s} \snorm{q,2}{u^n}. \label{b31-2} 
  \end{align} 

  Let us estimate $ \theta_h^n $. By \eqref{d4}, we have 
  \begin{align} 
    & \norm{s, 2, \ast}{\theta_h^n} \\ 
    & = \norm{s,2,\ast}{(\mathcal{S}_{\Lambda, \Delta t}^{\mathrm{D}} u^{n - 1}) \circ {X_{\Delta t}^1} [f \circ u^n] - (\mathcal{S}_{\Lambda, \Delta t}^{\mathrm{D}} u^{n - 1}) \circ {X_{\Delta t}^1} [f \circ \mathcal{S}_{f, \Delta t}^{\mathrm{A}} (\mathcal{S}_{\Lambda, \Delta t}^{\mathrm{D}} u_h^{n - 1})]} \\ 
    & \leq C_{\mathrm{P}} (\norm{s,2,\ast}{u^n}, \norm{s,2,\ast}{\mathcal{S}_{f, \Delta t}^{\mathrm{A}} (\mathcal{S}_{\Lambda, \Delta t}^{\mathrm{D}} u_h^{n - 1})})  \Delta t \\ & \qquad \cdot \norm{s+1,\infty}{\mathcal{S}_{\Lambda, \Delta t}^{\mathrm{D}} u^{n - 1}} \norm{s,2,\ast}{u^n - \mathcal{S}_{f, \Delta t}^{\mathrm{A}} (\mathcal{S}_{\Lambda, \Delta t}^{\mathrm{D}} u_h^{n - 1})}. \label{cs7} 
  \end{align} 
  By \eqref{tau}, \eqref{cs4} and \eqref{cs5}, we have 
  \begin{align} 
    u^n - \mathcal{S}_{f, \Delta t}^{\mathrm{A}} (\mathcal{S}_{\Lambda, \Delta t}^{\mathrm{D}} u_h^{n - 1}) = \eta_h^n + \theta_h^n + \tau_{\Delta t}^n \Delta t. \label{cs6} 
  \end{align} 
  Applying \eqref{ad2i}, \eqref{b8} and \eqref{cs6} to \eqref{cs7}, we have 
  \begin{align} 
    \norm{s, 2, \ast}{\theta_h^n} & \leq C_{\mathrm{P}} (\norm{s,2,\ast}{u^n}, \norm{s,2,\ast}{u_h^{n - 1}}) \Delta t \norm{s+1,\infty}{u^{n - 1}} \norm{s,2,\ast}{\eta_h^n + \theta_h^n + \tau_{\Delta t}^n \Delta t}. 
  \end{align} 
  Thus, there exists a positive constant $ C_{\mathrm{P}}^{(7),n} $ which depends polynomially on $ \norm{s,2,\ast}{u^n} $, $ \norm{s+1,\infty}{u^{n - 1}} $ and $ \norm{s,2,\ast}{u_h^{n - 1}} $ such that, if $ \Delta t \leq \Delta t_{h,1}^n $, 
  \begin{align} 
    \norm{s, 2, \ast}{\theta_h^n} & \leq C_{\mathrm{P}}^{(7),n} \Delta t \left(\norm{s,2,\ast}{\eta_h^n} + \norm{s,2,\ast}{\theta_h^n} + \Delta t \norm{s,2,\ast}{\tau_{\Delta t}^n}\right). \label{b13-2} 
  \end{align} 
  We define $ \Delta t_{h,2}^n = \min \{\Delta t_{h,1}^n, 1/(2 C_{\mathrm{P}}^{(7),n})\} $. If $ \Delta t \leq \Delta t_{h,2}^n $, then we have 
  \begin{align} 
    \frac{1}{1 - C_{\mathrm{P}}^{(7),n} \Delta t} \leq \frac{1}{2}. \label{dt2} 
  \end{align} 
  Using \eqref{dt2}, \eqref{b13-2} becomes 
  \begin{align} 
    \norm{s, 2, \ast}{\theta_h^n} & \leq \frac{1}{2} C_{\mathrm{P}}^{(7),n} \Delta t \left(\norm{s,2,\ast}{\eta_h^n} + \Delta t \norm{s,2,\ast}{\tau_{\Delta t}^n}\right). \label{b13} 
  \end{align} 
  In the rest of the proof, the dependence on $ \norm{L^{\infty} (W^{s + 1, \infty})}{u} $ will not be written explicitly. That is, even if $ C_{\mathrm{P}} $ depends on $ \norm{L^{\infty} (W^{s + 1, \infty})}{u} $ and $ \norm{s,2,\ast}{u_h^{n-1}} $, we simply denote it by $ C_{\mathrm{P}}(\norm{s,2,\ast}{u_h^{n-1}}) $. Applying \eqref{b13} to \eqref{b12}, we have 
  \begin{align} 
    \norm{s, 2, \ast}{e_h^n}^2 
    & \leq \norm{s, 2, \ast}{\eta_h^n}^2 
    + C_{\mathrm{P}} (\norm{s, 2, \ast}{u_h^{n - 1}} ) \Delta t \left(\norm{s, 2, \ast}{\eta_h^n}^2 + \norm{s, 2, \ast}{\tau_{\Delta t}^n}^2 \right) \\
    & \quad + \frac{C}{\Delta t} \norm{0,2}{\rho_h^n}^2 + \snorm{s,2}{\rho_h^n}^2. \label{b20-2} 
  \end{align} 
  Applying \eqref{b2}, \eqref{b31-1} and \eqref{b31-2} to \eqref{b20-2}, we obtain 
  \begin{align} 
    \norm{s, 2, \ast}{e_h^n}^2 & \leq \left(1 + C_{\mathrm{P}} (\norm{s, 2, \ast}{u_h^{n - 1}} ) \Delta t \right) \norm{s, 2, \ast}{e_h^{n - 1}}^2 \\ & \quad + C_{\mathrm{P}} (\norm{s, 2, \ast}{u_h^{n - 1}}) \Delta t \norm{s, 2, \ast}{\tau_{\Delta t}^{n}}^2 + C\left(\frac{h^{2q}}{\Delta t} + h^{2 (q - s)}\right)\snorm{q, 2}{u^n}^2. \label{b20} 
  \end{align} 
  By \eqref{ad3i} and \eqref{hyp:d}, \eqref{b20} becomes 
  \begin{align} 
    \norm{s, 2, \ast}{e_h^n}^2 & \leq \left(1 + C_{\mathrm{P}} (\norm{s, 2, \ast}{u_h^{n - 1}} ) \Delta t \right) \norm{s, 2, \ast}{e_h^{n - 1}}^2 \\ 
    & \quad + C_{\mathrm{P}} (\norm{s, 2, \ast}{u_h^{n - 1}}) K (u)^2 \Delta t^{1 + 2r} + C \Delta t^{1 + \varepsilon}\snorm{q, 2}{u^n}^2. \label{b18} 
  \end{align} 
  Let $ \bar{\varepsilon} = \min \{\varepsilon, 2r\} $. From \eqref{b18}, we deduce that 
  \begin{align} 
    & \norm{s, 2, \ast}{u_h^{n - 1}} \leq 2 \sup_{t\in (0, T)} \norm{s, 2, \ast}{u (t)} \\ 
    & \implies 
    \norm{s, 2, \ast}{e_h^n}^2 \leq \left(1 + \alpha_1 \Delta t \right) \norm{s, 2, \ast}{e_h^{n - 1}}^2 + \beta_1 \Delta t^{1 + \bar{\varepsilon}} \label{b15} 
  \end{align} 
  for some positive constants $ \alpha_1 $, depending on $ \norm{L^\infty(W^{s+1, \infty})}{u} $, 
  and $ \beta_1 $, depending on 
  $ \norm{L^\infty(W^{s+1, \infty})}{u} $, $ \norm{L^\infty (H^q)}{u} $ and $ K (u) $. 

  \paragraph{Step 2.} 
  In this step, we prove that there exist constants $ \Delta t_0 \in (0,1) $ and $ h_0 \in (0,1) $ such that 
  \begin{align} 
    & (h, \Delta t) \in (0, h_0) \times (0, \Delta t_0) \\
    & \implies \norm{s, 2, \ast}{u_h^n} \leq 2 \sup_{t\in (0, T)} \norm{s, 2, \ast}{u (t)} \quad \forall \, n = 0,\ldots,N_t. \label{b14} 
  \end{align} 

  Let $ (h_0, \Delta t_3) \in (0,1)^2 $ satisfy 
  \begin{align} 
    \beta_1 T \mathrm{e}^{\alpha_1 T} (\Delta t_3)^{\bar{\varepsilon}} \leq \frac{1}{2} \sup_{t\in (0, T)} \norm{s, 2, \ast}{u (t)}^2 \label{b16} 
  \end{align} 
  and 
  \begin{align} 
    (C_0)^2 \mathrm{e}^{\alpha_1 T} (h_0)^{2(q - s)} \leq \frac{1}{2} \sup_{t\in (0, T)} \norm{s, 2, \ast}{u (t)}^2. \label{b17} 
  \end{align} 

  We recall that $ \Delta t_{h,2}^n $ depends only on $ \snorm{1, \infty}{u_h^{n-1}} $, $ \norm{s,2,\ast}{u_h^{n-1}} $ and $ \norm{L^\infty (W^{s+1, \infty})}{u} $. Therefore, there exists a positive constant $ \Delta t_2 $, depending only on $ \norm{L^\infty(W^{s+1, \infty})}{u} $ and $ \norm{s + 1, \infty}{f} $, such that 
  \begin{align} 
    & \norm{s, 2, \ast}{u_h^{n - 1}} \leq 2 \sup_{t\in (0, T)} \norm{s, 2, \ast}{u (t)} 
    \implies \Delta t_{h,2}^n \leq \Delta t_2, \quad \forall \, n = 1,\ldots, N_t. 
  \end{align} 
  Let $ \Delta t_0 = \min \{\Delta t_2, \Delta t_3\} $. In the following, we assume $ (h, \Delta t) \in (0, h_0) \times (0, \Delta t_0) $. 

  We prove \eqref{b14} by induction. We first consider the case $ n = 0 $. By \eqref{b0} and \eqref{b17}, we have 
  \begin{align} 
    \norm{s,2,\ast}{e_h^0} \leq C_0 h^{q - s} \leq \frac{1}{\sqrt{2}} \sup_{t\in (0, T)}\norm{s, 2, \ast}{u (t)}. 
  \end{align} 
  Since $ u_h^0 = u_0 - e_h^0 $, we have 
  \begin{align} 
    \norm{s,2,\ast}{u_h^0} \leq \left(1 + \frac{1}{\sqrt{2}}\right) \sup_{t\in (0, T)}\norm{s, 2, \ast}{u (t)}, 
  \end{align} 
  which implies that
  \begin{align}
    \norm{s, 2, \ast}{u_h^n} \leq 2 \sup_{t\in (0, T)} \norm{s, 2, \ast}{u (t)} \label{b14-3} 
  \end{align}
  holds for $ n = 0 $. 

  Next, suppose that \eqref{b14-3} holds for all $ n = 0, \ldots, k $. By using \eqref{b15} recursively, we obtain 
  \begin{align}
    \norm{s, 2, \ast}{e_h^{k + 1}}^2 
    & \leq \left(1 + \alpha_1 \Delta t \right)^{k + 1} \norm{s, 2, \ast}{e_h^0}^2 + \sum_{n = 0}^{k} \left(1 + \alpha_1 \Delta t \right)^{n} \beta_1 \Delta t^{1 + \bar{\varepsilon}}. \label{b29} 
  \end{align} 
  Applying the inequalities
  \begin{gather}
    n \Delta t \leq T, \label{ea2} \\ 
    \left(1 + \alpha_1 \Delta t \right)^{n} \leq \mathrm{e}^{\alpha_1 T} \label{ea} 
  \end{gather}
  and \eqref{b0} to \eqref{b29}, we have 
  \begin{align} 
    \norm{s, 2, \ast}{e_h^{k + 1}}^2 & \leq (C_0)^2 \mathrm{e}^{\alpha_1 T} h^{2 (q - s)} + \beta_1 T \mathrm{e}^{\alpha_1 T} \Delta t^{\bar{\varepsilon}}. \label{b29-2} 
  \end{align} 
  Applying \eqref{b16} and \eqref{b17} to \eqref{b29-2}, we have 
  \begin{align} 
    \norm{s, 2, \ast}{e_h^{k + 1}}^2 \leq \sup_{t\in (0, T)} \norm{s, 2, \ast}{u (t)}^2, 
  \end{align} 
  which yields 
  \begin{align} 
    \norm{s, 2, \ast}{u_h^{k + 1}} & \leq \norm{s, 2, \ast}{u^{k + 1}} + \norm{s, 2, \ast}{e_h^{k + 1}} \\ 
    & \leq 2 \sup_{t\in (0, T)} \norm{s, 2, \ast}{u (t)}. 
  \end{align} 
  Therefore, we conclude that \eqref{b14} holds. 

  \paragraph{Step 3.} 
  We assume $ (h, \Delta t) \in (0, h_0) \times (0, \Delta t_0) $. As we have proved in Step 2, we have the bound \eqref{b14-3}. Then, \eqref{b20} becomes 
  \begin{align} 
    \norm{s, 2, \ast}{e_h^n}^2 & \leq \left(1 + C_{\mathrm{P}} (\norm{L^\infty (W^{s+1,\infty})}{u}) \Delta t \right) \norm{s, 2, \ast}{e_h^{n - 1}}^2 \\ & \quad + C_{\mathrm{P}} (\norm{L^\infty (W^{s+1,\infty})}{u}) \Delta t \norm{s, 2, \ast}{\tau_{\Delta t}^{n}}^2 + C\left(\frac{h^{2q}}{\Delta t} + h^{2 (q - s)}\right)\snorm{q, 2}{u^n}^2. \label{b20-3}
  \end{align} 
  Applying \eqref{ad3i} to \eqref{b20-3}, we have 
  \begin{align} 
    \norm{s, 2, \ast}{e_h^n}^2 & \leq \left(1 + \alpha_1 \Delta t \right) \norm{s, 2, \ast}{e_h^{n - 1}}^2 \\ & \quad + \beta_2 \left(K (u)^2 \Delta t^{1 + 2r} + \frac{h^{2q}}{\Delta t} + h^{2 (q - s)}\right), \label{b40} 
  \end{align} 
  where $ \alpha_1 $ is the constant introduced in \eqref{b15}, and $ \beta_2 $ is a constant which depends on $ \norm{L^\infty (W^{s+1,\infty})}{u} $ and $ \norm{L^{\infty}(H^{q})}{u} $. Under the condition \eqref{hyp:d}, we have $ h^{2q} / \Delta t \leq h^{2 (q - s)} $. Therefore \eqref{b40} becomes 
  \begin{align} 
    \norm{s, 2, \ast}{e_h^n}^2 & \leq \left(1 + \alpha_1 \Delta t \right) \norm{s, 2, \ast}{e_h^{n - 1}}^2 + \beta_2 \left(K (u)^2 \Delta t^{1 + 2r} + h^{2 (q - s)}\right). \label{b40-2} 
  \end{align} 
  Using \eqref{b40-2} recursively, we obtain 
  \begin{align} 
    \norm{s, 2, \ast}{e_h^n}^2 & \leq \left(1 + \alpha_1 \Delta t \right)^{n} \norm{s, 2, \ast}{e_h^0}^2 \\ & \quad + \beta_2 \sum_{k = 0}^{n - 1} \left(1 + \alpha_1 \Delta t \right)^k \left(K(u)^2 \Delta t^{1 + 2r} + h^{2 (q - s)}\right). \label{b40-3} 
  \end{align} 
  Applying \eqref{ea2}, \eqref{ea} and \eqref{b0} to \eqref{b40-3}, we have 
  \begin{align} 
    \norm{s, 2, \ast}{e_h^n}^2 & \leq C_0^2 \mathrm{e}^{\alpha_1 T} h^{2 (q - s)} + \beta T \mathrm{e}^{\alpha_1 T} \left(K(u)^2 \Delta t^{2r} + \frac{h^{2 (q - s)}}{\Delta t}\right), 
  \end{align} 
  which implies \eqref{b30}. 
\end{proof} 

\begin{proof}[Proof of Theorem~\ref{thm:2}] 
  In this proof, we do not write the dependence on $ \norm{s + 1, \infty}{f} $ explicitly. In the following, assume that $ \Delta t \leq 1 $. By (P\ref{p:s}) and (P\ref{p:q}), we have 
  \begin{align} 
    \norm{s, 2, \Delta}{e_h^0} & = \left(\norm{0,2}{(I - \mathcal{I}_{h}) u_0}^2 + \frac{h^{2s}}{\Delta t} \snorm{s,2}{(I - \mathcal{I}_{h}) u_0}^2\right)^{1/2} \\ 
    & \leq \left(C h^{2q} \snorm{q,2}{u_0}^2 + C \frac{h^{2q}}{\Delta t} \snorm{q,2}{u_0}^2\right)^{1/2} \\ 
    & \leq C \frac{h^q}{(\Delta t)^{1/2}} \snorm{q, 2}{u_0}. \label{b25} 
  \end{align} 
  Combining \eqref{b12-1} and \eqref{b12-2}, we have 
  \begin{align} 
    \norm{s, 2, \Delta}{e_h^n}^2 & = \norm{0, 2}{e_h^n}^2 + \frac{h^{2 s}}{\Delta t} \snorm{s, 2}{e_h^n}^2 \\ 
    & \leq (1 + 3 \Delta t) \norm{s, 2, \Delta}{\mathcal{I}_{h} \eta_h^n}^2 + \frac{h^{2s}}{\Delta t} \snorm{s, 2}{\rho_h^n}^2 \\ 
    & \quad + \frac{6}{\Delta t} \left(\norm{s, 2, \Delta}{\mathcal{I}_{h} \theta_h^n}^2 + \Delta t^2 \norm{s, 2, \Delta}{\mathcal{I}_{h} \tau_{\Delta t}^n}^2 + \norm{0, 2}{\rho_h^n}^2\right). 
  \end{align} 
  By \eqref{lem2}, we have 
  \begin{align} 
    \norm{s, 2, \Delta}{e_h^n}^2 & \leq (1 + C \Delta t) \norm{s, 2, \Delta}{\eta_h^n}^2 + \frac{C}{\Delta t} \norm{s, 2, \Delta}{\theta_h^n}^2 \\ 
    & \quad + 6 \Delta t \norm{s, 2, \Delta}{\tau_{\Delta t}^n}^2 + \frac{6}{\Delta t} \norm{0, 2}{\rho_h^n}^2 + \frac{h^{2s}}{\Delta t} \snorm{s, 2}{\rho_h^n}^2. \label{b24} 
  \end{align} 

  We estimate $ \eta_h^n $. By \eqref{c2-2} and \eqref{cs1}, we have 
  \begin{align} 
    \norm{s, 2, \Delta}{\eta_h^n} 
    & = \norm{s, 2, \Delta}{(\mathcal{S}_{\Delta t}^{\mathrm{V}} e_h^n \circ {X_{\Delta t}^1} [f \circ (\mathcal{S}_{f, \Delta t}^{\mathrm{A}}  \mathcal{S}_{\Delta t}^{\mathrm{V}} u_h^{n - 1})])} \\ 
    & \leq \left(1 + C_{\mathrm{P}} (\norm{s, 2, \ast}{\mathcal{S}_{f, \Delta t}^{\mathrm{A}}  \mathcal{S}_{\Delta t}^{\mathrm{V}} u_h^{n - 1}}) \Delta t\right) \norm{s, 2, \Delta}{\mathcal{S}_{\Delta t}^{\mathrm{V}} e_h^{n - 1}}. \label{b41} 
  \end{align}
  Applying \eqref{sv5} and \eqref{b8} to \eqref{b41}, we have 
  \begin{align}
    \norm{s, 2, \Delta}{\eta_h^n} \leq \left(1 + C_{\mathrm{P}} (\norm{s, 2, \ast}{u_h^{n - 1}}) \Delta t\right) \norm{s, 2, \Delta}{e_h^{n - 1}}. \label{b22-2} 
  \end{align} 

  Let $ h_0 $ and $ \Delta t_0 $ be the constants introduced in Step 2 of the proof of Theorem~\ref{thm:s}. In the following, we assume that $ (h, \Delta t) \in (0, h_0) \times (0, \Delta t_0) $. Then, we have the bound \eqref{b14-3}. Therefore, by \eqref{b22-2}, there exists a positive constant $ C_{\mathrm{P}}^{(8)} $ depending only on $ \sup_{t\in (0, T)} \norm{s, 2, \ast}{u^{n - 1}} $ such that 
  \begin{align} 
    \norm{s, 2, \Delta}{\eta_h^n} \leq \left(1 + C_{\mathrm{P}}^{(8)} \Delta t\right) \norm{s, 2, \Delta}{e_h^{n - 1}}. \label{b22} 
  \end{align} 

  Let us estimate $ \tau_{\Delta t}^n $. From \eqref{ad3i} and \eqref{hyp:d}, we deduce that 
  \begin{align} 
    \norm{s,2,\Delta}{\tau_{\Delta t}^n} & = \left(\norm{0,2}{\tau_{\Delta t}^n}^2 + \frac{h^{2s}}{\Delta t} \snorm{s,2}{\tau_{\Delta t}^n}^2\right)^{1/2} \\ 
    & \leq \left(1 + \frac{h^{2s}}{\Delta t}\right)^{1/2} \norm{s,2,\ast}{\tau_{\Delta t}^n} \\ 
    & \leq C K (u) \Delta t^r. \label{b32r} 
  \end{align} 

  We estimate $ \theta_h^n $. By \eqref{d22}, we have 
  \begin{align} 
    & \norm{s, 2, \Delta}{\theta_h^n} \\ 
    & = \norm{s, 2, \Delta}{(\mathcal{S}_{\Delta t}^{\mathrm{V}} u^{n - 1}) \circ {X_{\Delta t}^1}[f \circ u^n] - (\mathcal{S}_{\Delta t}^{\mathrm{V}} u^{n - 1}) \circ {X_{\Delta t}^1}[f \circ (\mathcal{S}_{f, \Delta t}^{\mathrm{A}}  \mathcal{S}_{\Delta t}^{\mathrm{V}} u_h^{n - 1})]} \\
    & \leq C_{\mathrm{P}} (\norm{s, 2, \ast}{u^n}, \norm{s, 2, \ast}{\mathcal{S}_{f, \Delta t}^{\mathrm{A}}  \mathcal{S}_{\Delta t}^{\mathrm{V}} u_h^{n - 1}}) \Delta t \norm{s + 1, \infty}{\mathcal{S}_{\Delta t}^{\mathrm{V}} u^{n - 1}} \\ & \qquad \cdot \norm{s, 2, \Delta}{u^n - \mathcal{S}_{f, \Delta t}^{\mathrm{A}}  \mathcal{S}_{\Delta t}^{\mathrm{V}} u_h^{n - 1}}. \label{b42}
  \end{align}
  Applying \eqref{ad2i}, \eqref{b8} and \eqref{cs6} to \eqref{b42}, we have 
  \begin{align} 
    \norm{s, 2, \Delta}{\theta_h^n} & \leq C_{\mathrm{P}} (\norm{L^\infty(W^{s+1, \infty})}{u}, \norm{s, 2, \ast}{u_h^{n - 1}}) \Delta t \\ & \qquad \cdot (\norm{s, 2, \Delta}{\eta_h^n} + \norm{s, 2, \Delta}{\theta_h^n} + \Delta t \norm{s, 2, \Delta}{\tau_{\Delta t}^n}). 
  \end{align} 
  From the bound \eqref{b14-3} for the numerical solution, we deduce that there exists a positive constant $ C^{(9)} $, independent of the numerical solution $ u_h^n $, such that 
  \begin{align} 
    \norm{s, 2, \Delta}{\theta_h^n} \leq C^{(9)} \Delta t (\norm{s, 2, \Delta}{\eta_h^n} + \norm{s, 2, \Delta}{\theta_h^n} + \Delta t \norm{s, 2, \Delta}{\tau_{\Delta t}^n}). \label{b21} 
  \end{align} 
  We note that $ C^{(9)} $ depend on $\|u\|_{L^\infty(W^{s+1,\infty}(\mathbb{T}))}$. Let $ \Delta t_0^\prime = \min \{1 / (2 C^{(9)}), \Delta t_0 \} $. In the rest of this proof, assume that $ \Delta t \leq \Delta t_0^\prime $. Then we have 
  \begin{align} 
    \frac{1}{1 - C^{(9)} \Delta t} \leq \frac{1}{2}. \label{b23-2} 
  \end{align} 
  Applying \eqref{b23-2} to \eqref{b21}, we have 
  \begin{align} 
    \norm{s, 2, \Delta}{\theta_h^n} \leq \frac{1}{2} C^{(9)} \Delta t (\norm{s, 2, \Delta}{\eta_h^n} + \Delta t \norm{s, 2, \Delta}{\tau_{\Delta t}^n}). \label{b23}
  \end{align} 
  Using \eqref{b23}, \eqref{b24} becomes 
  \begin{align} 
    \norm{s, 2, \Delta}{e_h^n}^2 & \leq (1 + C \Delta t) \norm{s, 2, \Delta}{\eta_h^n}^2 
    + C \Delta t \norm{s, 2, \Delta}{\tau_{\Delta t}^n}^2 \\ & \quad + \frac{6}{\Delta t} \norm{0, 2}{\rho_h^n}^2 + \frac{h^{2s}}{\Delta t} \snorm{s, 2}{\rho_h^n}^2. \label{b24-3} 
  \end{align} 
  By \eqref{b22} and this, there exists a positive constant $ \alpha_2 $, depending only on $ \norm{L^\infty (W^{s+1, \infty})}{u} $, such that 
  \begin{align} 
    \norm{s, 2, \Delta}{e_h^n}^2 & \leq (1 + \alpha_2 \Delta t) \norm{s, 2, \Delta}{e_h^{n - 1}}^2 + C \Delta t \norm{s, 2, \Delta}{\tau_{\Delta t}^n}^2 \\ & \quad + \frac{6}{\Delta t} \norm{0, 2}{\rho_h^n}^2 + \frac{h^{2s}}{\Delta t} \snorm{s, 2}{\rho_h^n}^2. \label{b24-2} 
  \end{align} 
  Applying \eqref{b31-1}, \eqref{b31-2} and \eqref{b32r} to \eqref{b24-2}, we obtain 
  \begin{align}
    \norm{s, 2, \Delta}{e_h^n}^2 
    & \leq (1 + \alpha_2 \Delta t) \norm{s, 2, \Delta}{e_h^{n - 1}}^2 + \beta_3 \left(K (u)^2 \Delta t^{1 + 2 r} + \frac{h^{2 q}}{\Delta t}\right), \label{b43}
  \end{align} 
  where $ \beta_3 $ is a positive constant depending on $ \norm{L^\infty (W^{s+1,\infty})}{u} $ and $ \norm{L^{\infty}(H^{q})}{u} $. Using \eqref{b43} repeatedly, we have 
  \begin{align} 
    \norm{s, 2, \Delta}{e_h^n}^2 \leq (1 + \alpha_2 \Delta t)^n \norm{s, 2, \Delta}{e_h^0}^2 + \sum_{k=0}^{n-1} (1 + \alpha_2 \Delta t)^k \beta_3 \left(K(u)^2 \Delta t^{1 + 2 r} + \frac{h^{2 q}}{\Delta t}\right). 
  \end{align} 
  By \eqref{b25} and $ (1 + \alpha_2 \Delta t)^n \leq e^{\alpha_2 T} $, we have 
  \begin{align} 
    \norm{s, 2, \Delta}{e_h^n}^2 \leq e^{\alpha_2 T} \left(C \frac{h^{2q}}{\Delta t} \snorm{q, 2}{u_0}^2 + \beta_3 T \left(K (u)^2 \Delta t^{2 r} + \frac{h^{2q}}{\Delta t^2}\right)\right), 
  \end{align} 
  which implies \eqref{b26}. 
\end{proof} 

\section{Numerical experiments}\label{sec:test}

We performed numerical experiments for the initial value problem \eqref{ivp1}--\eqref{ivp2} with $ f (u) = u $. 
We define the incomplete Jacobi elliptic integral of first kind by 
\begin{align} 
  F (\varphi, k) = \int_{0}^{\varphi} \frac{\mathrm{d}\theta}{\sqrt{1 - k^2 \sin^2 \theta}} 
\end{align} 
and define the complete Jacobi elliptic integral of first kind by $ K (k) = F (\pi/2, k) $. Let $ \varphi_k: \mathbb{R} \to \mathbb{R} $ be the inverse of $ F (\cdot, k) $. We define the Jacobi elliptic function $ \cn (x,k) $ by 
\begin{align} 
  \cn (x, k) = \cos \varphi_k (x). 
\end{align} 

Let $ \nu = 10^{-3} $. Let the initial value be given by 
\begin{align} 
  u_0 (x) = \frac{1}{10} + \frac{3 \nu}{2 K (1/\sqrt{2})} \cn (2 K (1/\sqrt{2}) x, 1/\sqrt{2}). 
\end{align} 
Then, the exact solution to \eqref{ivp1}--\eqref{ivp2} can be written as 
\begin{align} 
  u_0 (x) = \frac{1}{10} + \frac{3 \nu}{2 K (1/\sqrt{2})} \cn (2 K (1/\sqrt{2}) (x - t / 10), 1/\sqrt{2}). 
\end{align} 
We define the relative $ L^2 $-error by 
\begin{align} 
  \mathcal{E}_h^{N_t} = \frac{\norm{0,2}{u^{N_t} - u_h^{N_t}}}{\norm{0,2}{u^{N_t}}}. 
\end{align} 
We compute the $ L^2 $-norms approximately using the Legendre--Gauss quadrature with seven nodes. In all the numerical experiments below, we used the cubic Hermite interpolation for $ \mathcal{I}_h $ in \eqref{sl1}--\eqref{sl2}. 

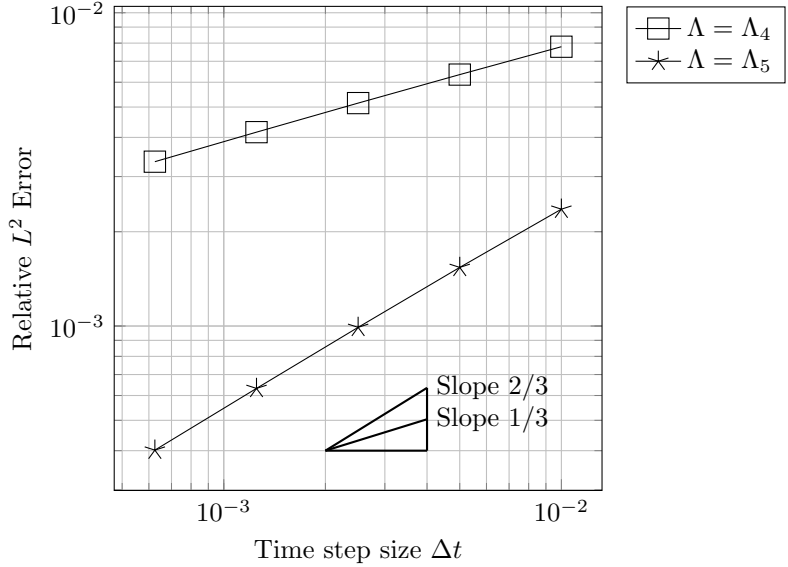
\begin{figure} 
  \centering 
  \begin{tikzpicture} 
    \begin{loglogaxis}[ 
      clip=false, 
      xlabel={Time step size $ \Delta t $}, 
      ylabel={Relative $ L^2 $ Error}, 
      grid=both, 
      width=8cm, 
      height=8cm, 
      legend style={at={(1.05,1)}, anchor=north west} 
    ]
      \addplot[
        mark=square, 
        mark size=4, 
      ] coordinates { 
        (1 /  100, 0.00779164) 
        (1 /  200, 0.00634823) 
        (1 /  400, 0.00514792) 
        (1 /  800, 0.00415775) 
        (1 / 1600, 0.00334669) 
      }; 
      \addlegendentry{$ \Lambda = \Lambda_4 $} 
      \addplot[
        mark=star, 
        mark size=4, 
      ] coordinates { 
        (1 /  100, 0.00236038 ) 
        (1 /  200, 0.00153862 ) 
        (1 /  400, 0.00099024 ) 
        (1 /  800, 0.000632204) 
        (1 / 1600, 0.000401607) 
      }; 
      \addlegendentry{$ \Lambda = \Lambda_5 $} 
      \draw[thick] (axis cs:\pxa,  \pya) -- (axis cs:{\pxa *2},{\pya *pow(2,1/3)}); 
      \draw[thick] (axis cs:\pxa,  \pya) -- (axis cs:{\pxa *2},{\pya *pow(2,2/3)}); 
      \draw[thick] (axis cs:\pxa,  \pya) -- (axis cs:{\pxa *2},{\pya}); 
      \draw[thick] (axis cs:{\pxa *2},{\pya *pow(2,2/3)}) -- (axis cs:{\pxa *2},{\pya}); 
      \node[anchor=west] at (axis cs:{\pxa *2},{\pya *pow(2,1/3)}) {Slope $ 1/3 $}; 
      \node[anchor=west] at (axis cs:{\pxa *2},{\pya *pow(2,2/3)}) {Slope $ 2/3 $}; 
    \end{loglogaxis} 
  \end{tikzpicture} 
  \caption{Relative $ L^2 $-error versus the time step $ \Delta t $. The mesh is uniform in both space and time, and the spatial mesh size is fixed at $ h = 1 / 1000$. The $ L^2 $-errors are evaluated at $ T = 1 $. } 
  \label{fig:1} 
\end{figure} 
\begin{figure} 
  \centering 
  \begin{tikzpicture} 
    \begin{loglogaxis}[ 
      clip=false, 
      xlabel={Spatial mesh size $ h $}, 
      ylabel={Relative $ L^2 $ Error}, 
      grid=both, 
      width=8cm, 
      height=6cm, 
      legend style={at={(1.05,1)}, anchor=north west} 
    ] 
      \addplot[ 
        mark=star, 
        mark size=4, 
      ] coordinates {  
        (1 /  16,  0.00864922) 
        (1 /  32,  0.00391716) 
        (1 /  64,  0.00146543) 
        (1 / 128, 0.000501557) 
        (1 / 256, 0.000167431) 
        (1 / 512, 5.54486e-05) 
      }; 
      \addlegendentry{$ \Lambda = \Lambda_5 $} 
      \draw[thick] (axis cs:\pxb,  \pyb) -- (axis cs:{\pxb *2},{\pyb *pow(2,8/5)}); 
      \draw[thick] (axis cs:\pxb,  \pyb) -- (axis cs:{\pxb *2},{\pyb}); 
      \draw[thick] (axis cs:{\pxb *2},{\pyb *pow(2,8/5)}) -- (axis cs:{\pxb *2},{\pyb}); 
      \node[anchor=west] at (axis cs:{\pxb *2},{\pyb *pow(2,8/5)}) {Slope $ 8/5 $}; 
    \end{loglogaxis} 
  \end{tikzpicture} 
  \caption{Relative $ L^2 $-error versus the spatial mesh size $ h $. The mesh is uniform in both space and time. The time step size is given by $ \Delta t = 100 h^{12/5} $. The $ L^2 $-norm is evaluated at $ T = N_t \Delta t $, where $ N_t = \lfloor 1 / \Delta t \rfloor $. } 
  \label{fig:2} 
\end{figure}
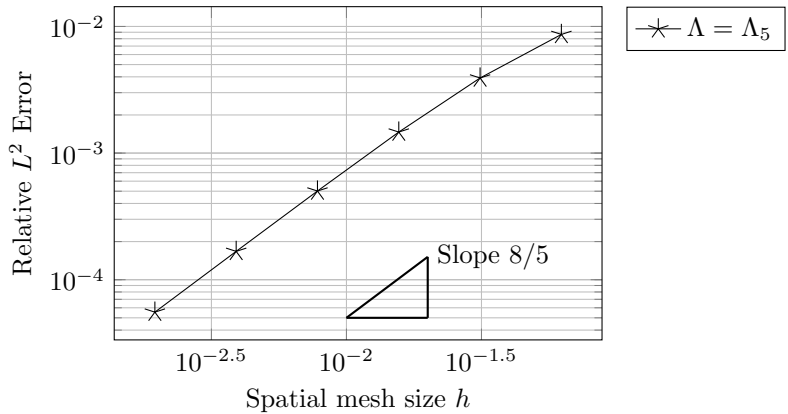 

Figure \ref{fig:1} shows the graph of the relative $ L^2 $-error versus the time step size $ \Delta t $. In these experiments, the spatial mesh size is fixed. We observe that the $ L^2 $-errors of the schemes with $ \Lambda = \Lambda_4 $ and $ \Lambda = \Lambda_5 $ are approximately $ O(\Delta t^{1/3}) $ and $ O(\Delta t^{2/3}) $, respectively. There convergence rates also appear in the theoretical error estimates. 

In Figure \ref{fig:2}, we tested the scheme with $ \Lambda = \Lambda_5 $. The discretization parameters $ h $ and $ \Delta t $ are given by $ h = 16, 32, \ldots, 512 $ and $ \Delta t = 100 h^{12/5} $. The integer $ q $ in Theorem \ref{thm:2} is $ 4 $ when we use the cubic Hermite interpolation. Thus, the error bound in Theorem \ref{thm:2} can be written as $ \mathcal{E}_h^{N_t} = O (\Delta t^{2/3} + h^{4} / \Delta t) $. We note that when $ \Delta t $ is proportional to $ h^{12/5} $, $ \Delta t^{2/3} $ is proportional to $ h^{4} / \Delta t $. Therefore, in this setting, the theoretical error bound is given by $ \mathcal{E}_h^{N_t} = O (h^{8/5}) $. The results shown in the figure are consistent with this theoretical error estimate. 

\providecommand{\bysame}{\leavevmode\hbox to3em{\hrulefill}\thinspace}
\providecommand{\MR}{\relax\ifhmode\unskip\space\fi MR }
\providecommand{\MRhref}[2]{%
  \href{http://www.ams.org/mathscinet-getitem?mr=#1}{#2}
}
\providecommand{\href}[2]{#2}

\end{document}